\documentclass[11pt]{amsart}

\usepackage{amsmath}
\usepackage{amssymb}
\usepackage{amscd}
\usepackage{latexsym}
\usepackage[latin1]{inputenc}
\usepackage{enumerate}
\usepackage[all]{xy}
\usepackage{hyperref}

\usepackage{graphicx}
\usepackage{changebar,color}
\usepackage{fullpage}

\newtheorem{theorem}{Theorem}[section]
\newtheorem{corollary}[theorem]{Corollary}
\newtheorem{lemma}[theorem]{Lemma}

\newtheorem{remark}[theorem]{Remark}
\newtheorem{proposition}[theorem]{Proposition}

\newtheorem{definition}[theorem]{Definition}

\numberwithin{equation}{section}

\newcommand{\Ker}{\operatorname{Ker}}

\def\Z{\mathbb{Z}}

\def\Q{\mathbb{Q}}
\def\R{\mathbb{R}}
\def\C{\mathbb{C}}

\def\Fp{\mathbb{F}_{p}}

\newcommand{\cX}{{\mathcal X}}
\newcommand{\cY}{{\mathcal Y}}
\newcommand{\cZ}{{\mathcal Z}}

\newcommand{\OK}{{\mathcal O}_K}

\newcommand{\Spec}{{\rm Spec\, }}

\newcommand{\Lbar}{{\overline{\mathcal L}}}
\newcommand{\Mbar}{{\overline{\mathcal M}}}

\newcommand{\Bbar}{{\overline{\mathcal B}}}
\newcommand{\Ebar}{{\overline{\mathcal E}}}
\newcommand{\Abar}{{\overline{\mathcal A}}}
\newcommand{\Obar}{{\overline{\mathcal O}}}

\newcommand{\ra}{\rightarrow}

\title{Arithmetic ampleness and an arithmetic Bertini theorem}
\author{Fran\c cois Charles}
\address{Fran\c cois Charles, Laboratoire de mathématiques d'Orsay, UMR 8628 du CNRS, Universit{é}
Paris-Sud,
B{â}timent 425, 91405 Orsay cedex, France}
\email{francois.charles@math.u-psud.fr}
\begin{document}

\maketitle

\begin{abstract}
Let $\cX$ be a projective arithmetic variety of dimension at least $2$. If $\Lbar$ is an ample hermitian line bundle on $\cX$, we prove that the proportion of those effective sections of $\Lbar^{\otimes n}$ that define an irreducible divisor on $\cX$ tends to $1$ as $n$ tends to $\infty$. We prove variants of this statement for schemes mapping to such an $\cX$.

On the way to these results, we discuss some general properties of arithmetic ampleness, including restriction theorems, and upper bounds for the number of effective sections of hermitian line bundles on arithmetic varieties.
\end{abstract}


\section{Introduction}

\subsection{Bertini theorems over fields}
Let $k$ be an infinite field, and let $X$ be an irreducible variety over $k$ with dimension at least $2$. Given an embedding of $X$ in some projective space over $k$, the classical Bertini theorem \cite[Theorem 3.3.1]{Lazarsfeld04} shows, in its simplest form, that infinitely many hyperplane sections of $X$ are irreducible. 

In the case where $k$ is finite, the Bertini theorem can fail, since the finitely many hyperplane sections of $X$ can all be reducible. As was first explained in \cite{Poonen04} in the setting of smoothness theorems, this phenomenon can be dealt with by replacing hyperplane sections with ample hypersurfaces of higher degree. We can state the main result of \cite{CharlesPoonen16} -- see Theorem 1.6 and the discussion in the proof of Theorem \ref{theorem:finite-fields} -- as follows: let $k$ be a finite field, let $X$ be a projective variety over $k$ and let $L$ be an ample line bundle on $X$. Let $Y$ be a scheme of finite type over $k$ together with a morphism $f : Y\ra X$. Assume that the image of $f$ is irreducible of dimension at least $2$. If $Z$ is a subscheme of $Y$, write $Z_{horiz}$ for the union of those irreducible components of $Z$ that do not map to a closed point of $X$. Then the set 
$$\mathcal P= \{\sigma\in \bigcup_{n>0} H^0(X, L^{\otimes n}), \,\mathrm{div}(f^*\sigma)_{horiz}\,\textup{is irreducible}\}$$
has density $1$, in the sense that 
$$\lim_{n\ra\infty}\frac{|\mathcal P\cap H^0(X, L^{\otimes n})|}{|H^0(X, L^{\otimes n})|}=1.$$
When $Y$ is a subscheme of $X$, we can disregard the horizontality subscript.

\subsection{The arithmetic case}

In this paper, we deal with an arithmetic version of Bertini theorems as above. Let $\cX$ be an arithmetic variety, that is, an integral scheme which is separated, flat of finite type over $\Spec\Z$. Assume that $\cX$ is projective, and let $\mathcal L$ be a relatively ample line bundle on $\cX$. As is well known, sections of $\mathcal L$ over $\cX$ are not the analogue of global sections of a line bundle over a projective variety over a field. Indeed, it is more natural to consider a hermitian line bundle $\Lbar$ with underlying line bundle $\mathcal L$ and consider the sets
$$H^0_{Ar}(\cX, \Lbar^{\otimes n})$$
of sections with norm at most $1$ everywhere. We discuss ampleness for hermitian line bundles in section \ref{section:ampleness}, which we refer to for definitions.

Given finite sets $(X_n)_{n>0}$, and a subset $\mathcal P$ of $\bigcup_{n>0}X_n$, say that $\mathcal P$ has density $\rho$ if the following equality holds:
$$\lim_{n\ra\infty}\frac{|\mathcal P\cap X_n|}{|X_n|}=\rho.$$

The main result of this paper is the following arithmetic Bertini theorem. Again, given a morphism of schemes $f : Y\ra X$, and if $Z$ is a subscheme of $Y$, we denote by $Z_{horiz}$ the union of those irreducible components of $Z$ that do not map to a closed point of $X$. If $\Lbar=(\mathcal L, ||.||)$ is a hermitian line bundle and $\delta$ is a real number, write $\Lbar(\delta)$ for the hermitian line bundle $(\mathcal L, e^{-\delta}||.||)$.

\begin{theorem}\label{thm:main}
Let $\cX$ be a projective arithmetic variety, and let $\Lbar$ be an ample hermitian line bundle on $\cX$. Let $\cY$ be an integral scheme of finite type over $\Spec\Z$ together with a morphism $f : \cY\ra \cX$ which is generically smooth over its image. Assume that the image of $Y$ has dimension at least $2$. Let $\varepsilon>0$ be a real number. Then the set 
$$\{\sigma\in\bigcup_{n> 0} H^0_{Ar}(\cX, \Lbar(\varepsilon)^{\otimes n}), \, ||\sigma_{|f(\cY(\C))}||_\infty<1 \,\textup{and div}(f^*\sigma)_{horiz}\,\textup{is irreducible}\}$$
has density $1$ in $\{\sigma \in\bigcup_{n> 0} H^0_{Ar}(\cX, \Lbar(\varepsilon)^{\otimes n}), \, ||\sigma_{|f(\cY(\C))}||_\infty<1\}.$
\end{theorem}

\begin{remark}
A Weil divisor is said to be irreducible if it comes from an integral codimension $1$ subscheme. 
\end{remark}

\begin{remark}
The hypothesis that $f$ is generically smooth over its image is necessary: when $f$ is the Frobenius morphism of a fiber $\cX_p$, all $\mathrm{div}(f^*\sigma)$ have components with multiplicities divisible by $p$. Of course, it holds when $\cY$ is flat over $\Spec\Z$. Without this hypothesis on $f$, the conclusion is only that the support of $\mathrm{div}(f^*\sigma)$ is irreducible for a density $1$ set of $\sigma$. \end{remark}

An important special case of the theorem deals with the special case where $f$ is dominant. 

\begin{theorem}\label{thm:main-dominant}
Let $\cX$ be a projective arithmetic variety, and let $\Lbar$ be an ample hermitian line bundle on $\cX$. Let $\cY$ be an integral scheme of finite type over $\Spec\Z$ together with a morphism $f : \cY\ra \cX$. Assume that the image of $\cY$ has dimension at least $2$ and $f$ is dominant, generically smooth over its image. Then the set 
$$\{\sigma\in\bigcup_{n> 0} H^0_{Ar}(\cX, \Lbar^{\otimes n}), \,\mathrm{div}(f^*\sigma)_{horiz}\,\textup{is irreducible}\}$$
has density $1$ in $\bigcup_{n> 0} H^0_{Ar}(\cX, \Lbar^{\otimes n})$.
\end{theorem}

The case where $\cY=\cX$ is particularly significant. We state it independently below. Most of this paper will be devoted to its proof, and we will prove \ref{thm:main} and \ref{thm:main-dominant} as consequences.

\begin{theorem}\label{thm:X=Y}
Let $\cX$ be a projective arithmetic variety of dimension at least $2$, and let $\Lbar$ be an ample hermitian line bundle on $\cX$. Then the set 
$$\{\sigma\in\bigcup_{n> 0} H^0_{Ar}(\cX, \Lbar^{\otimes n}), \,\mathrm{div}(\sigma)\,\textup{is irreducible}\}$$
has density $1$ in $\bigcup_{n> 0} H^0_{Ar}(\cX, \Lbar^{\otimes n})$.
\end{theorem}

Theorem \ref{thm:X=Y} is stronger than the Bertini irreducibility theorem of \cite[Theorem 1.1]{CharlesPoonen16}, as we explain in Section 3. Note however that we use the results of \cite{CharlesPoonen16} in our proofs.

In Theorem \ref{thm:X=Y}, the case where $\cX$ has dimension at least $3$ -- that is, relative dimension at least $2$ over $\Spec\Z$ -- is easier. Indeed, if $p$ is a large enough prime number, we can apply the Bertini irreducibility theorems over finite fields to the reduction of $\cX$ modulo $p$, which with moderate work is enough to prove the theorem. However, when $\cX$ is an arithmetic surface, Theorem \ref{thm:X=Y} is genuinely different from its finite field counterpart. Note that the hardest case of the main result of \cite{CharlesPoonen16} is the surface case as well.

Theorem \ref{thm:main} should be compared to the Hilbert irreducibility theorem, which implies, if $\mathcal L$ is very ample on the generic fiber of $\cX$ and $Y$ is flat over $\Z$, the existence of many sections $\sigma$ of $\mathcal L$ such that the generic fiber of div($f^*\sigma)$ is irreducible. However, the Hilbert irreducibility theorem does not guarantee that these sections have small norm. To our knowledge, Theorem \ref{thm:main} does not imply the Hilbert irreducibility theorem, nor does it follow from it.

\bigskip

Arithmetic Bertini theorems, in the setting of both general arithmetic geometry and Arakelov geometry, have appeared in the literature. In \cite{Poonen04}, Poonen is able to prove a Bertini regularity theorem for ample line bundles on regular quasi-projective schemes over $\Spec\Z$ under the abc conjecture and technical assumptions. The statement does not involve hermitian metrics but still involves a form of density.

In \cite{Moriwaki95}, Moriwaki proves a Bertini theorem showing the existence of at least one effective section of large powers of an arithmetically ample line bundle that defines a generically smooth divisor -- this was reproved and generalized in \cite{Ikoma15}. As a byproduct of our discussion of arithmetic ampleness in section 2 and Poonen's result over finite fields, we will give a short proof of a more precise version of this result.

\begin{theorem}\label{theorem:generic-smoothness}
Let $\cX$ be a projective arithmetic variety with smooth generic fiber, and let $\Lbar$ be an ample hermitian line bundle on $\cX$. Then the set 
$$\{\sigma\in\bigcup_{n> 0} H^0_{Ar}(\cX, \Lbar^{\otimes n}), \,\mathrm{div}(\sigma)_\Q\,\textup{is smooth}\}$$
has density $1$ in $\bigcup_{n> 0} H^0_{Ar}(\cX, \Lbar^{\otimes n})$.
\end{theorem}
Of course, this result can be combined with Theorem \ref{thm:X=Y} if $\cX$ has dimension at least $2$. 

Weaker Bertini theorems over rings of integers in number fields have been used in higher class-field theory, under the form of the Bloch-Raskind-Kerz approximation lemma proved in \cite{Bloch81, Raskind95, Wiesend06, Kerz11} -- see \cite[Lemme 5.2]{Szamuely10} for a discussion. These results can be obtained easily as a special case of our Corollary \ref{corollary:Taylor} (or its variant corresponding to Theorem \ref{thm:main} for Wiesend's version) -- this corollary allows us furthermore to control the cohomology class of the irreducible subvarieties involved. 

An arithmetic Bertini theorem has been proved by Autissier in \cite{Autissier01, Autissier02}. Counts of irreducible divisors on arithmetic varieties have been provided by many authors, starting with Faltings in \cite{Faltings84}.

\subsection{Strategy of the proofs}
The starting point of our proofs is that, as in \cite{CharlesPoonen16}, the Bertini irreducibility theorem is susceptible to a counting approach: to show that most divisors are irreducible, simply bound the number of the reducible ones. 

To carry on this approach, we need to translate in Arakelov geometry results from classical geometry. The two key results in that respect are the study of restriction maps for powers of ample hermitian line bundles we prove in \ref{subsection:restriction} and bounds for sections of hermitian line bundles on surfaces in \ref{subsection:upper-bound-theta}. We hope that these results have independent interest.

Even with these tools at our disposal, we are not able to adapt the methods of \cite{CharlesPoonen16}, for two reasons. First, the error terms in the various estimates we deal with (including arithmetic Hilbert-Samuel) are big enough that we need a more involved sieving technique than in \cite{CharlesPoonen16} involving the analysis of simultaneous restriction of sections modulo infinite families of subschemes. Second, given a section of a hermitian line bundle with reducible divisor on a regular arithmetic surface, we need to construct a corresponding decomposition of the hermitian line bundle, which involves constructing suitable metrics. The relevant analysis is dealt with in \ref{subsection:norm-estimates} and can only be applied when sufficient geometric bounds hold. To get a hold of the geometry, we need a careful analysis dealing with infinite families of curves over finite fields -- coming from the reduction of our given arithmetic surface modulo many primes. This is the content of \ref{subsection:finite-fields}.

\subsection{Notation and definitions} 
By an arithmetic variety, we mean an integral scheme which is flat of finite type over $\Spec\Z$. A projective arithmetic variety of dimension $2$ is an arithmetic surface. If $\cX$ is a scheme over $\Spec\Z$ and if $p$ is a prime number, we will denote by $\cX_p$ the reduction of $\cX$ modulo $p$. If $f : X\ra Y$ is a morphism of noetherian schemes, we say that an irreducible component of $X$ is vertical if its image is a closed point of $Y$, and horizontal if not. We denote by $X_{horiz}$ the union of the horizontal components of $X$.

If $\cX$ is a scheme of finite type over $\Spec\Z$ with reduced generic fiber -- in particular, if $\cX$ is an arithmetic variety -- a hermitian line bundle $\Lbar$ on $\cX$ is a pair $(\mathcal L, ||.||)$ consisting of a line bundle $\mathcal L$ on $\cX$, and a smooth hermitian metric $||.||$ on $\mathcal L_\C$ which is invariant under complex conjugation. If $\Lbar$ is a hermitian line bundle over an arithmetic variety $\cX$, we will denote by $\mathcal L$ the underlying line bundle. Note that if the generic fiber $\cX_\Q$ is empty, a hermitian line bundle on $\cX$ is nothing but a line bundle.

Let $\Lbar$ be a hermitian line bundle on a proper scheme $\cX$ over $\Z$ with reduced generic fiber. If $\sigma$ is a section of $\mathcal L$ on $\cX$, we will often denote by $||\sigma||$ the sup norm of $\sigma$. If $||\sigma||\leq 1$ (resp. $||\sigma||<1$), we say that $\sigma$ is effective (resp. strictly effective). We denote by $H^0_{Ar}(\cX, \Lbar)$ the finite set of effective sections of $\Lbar$, and write 
$$h^0_{Ar}(\cX, \Lbar)=\log |H^0_{Ar}(\cX, \Lbar)|.$$
Note again that if $\cX_\Q$ is empty, then 
$$H^0_{Ar}(\cX, \Lbar)=H^0(\cX, \mathcal L).$$
We say that a hermitian line bundle on $\cX$ is effective if it has a nonzero effective section.

\subsection{Outline of the paper}

Section \ref{section:ampleness} is devoted to a general discussion of arithmetic ampleness. After setting definitions, we recall aspects of the arithmetic Hilbert-Samuel theorem, taking care of error terms. We then prove a number of results concerning the image of restriction maps for sections of large powers of ample hermitian line bundles. 

In the short section 3, we make use of the previous section to discuss consequences and variants of the main theorems. We prove Theorem \ref{theorem:generic-smoothness}. 

In section 4, we gather general results -- both analytic and arithmetic -- dealing with hermitian line bundles on Riemann surfaces and arithmetic surfaces. We prove norm estimates in the spirit of \cite{BostGilletSoule94}, and we prove a basic upper bound on the number of effective sections for positive line bundles -- in some sense -- on arithmetic surfaces, making use of the $\theta$-invariants of Bost, as well as a result on the effective cone of arithmetic surfaces. 

Section 5 is devoted to the proof of Theorem \ref{thm:X=Y}, and section 6 to the remaining theorems of the introduction.

\subsection{Acknowledgements} This paper arose from a question raised by Hrushovski at the SAGA in Orsay. It is a great pleasure to thank him for the impetus for this work. Working with Bjorn Poonen on \cite{CharlesPoonen16} has obviously been crucial to this paper, and we gratefully acknowledge his influence. While writing this paper, we benefited from a very large number of conversations with Jean-Benoît Bost whom we thank heartfully -- it should be clear to the reader that this paper owes him a large intellectual debt. We thank Jean-Louis Colliot-Thélène and Offer Gabber for helpful remarks.

This project has received funding from the European Research Council (ERC) under the European Union's Horizon 2020 research and innovation programme (grant agreement No 715747).

\section{Some results on arithmetic ampleness}\label{section:ampleness}

In this section, we gather some results on ample line bundles on arithmetic varieties. Most of the results are well-known and can be found in a similar form in the literature. Aside from a precise statement regarding error terms in the arithmetic Hilbert-Samuel theorem, our main original contribution consists in the results of \ref{subsection:restriction} that deals with restriction maps for sections of ample line bundles.

\subsection{Definitions and basic properties} We give a definition of arithmetic ampleness which is a little bit more restrictive than usual. 

\begin{definition}
Let $X$ be a complex analytic space, and let $\overline L=(L, ||.||)$ be a hermitian line bundle on $X$. Let $\mathbb V(L)$ be the total space of $L$, and $Z\subset \mathbb V(L)$ be the zero section. We say that $\overline L$ is \emph{positive} if the function 
$$(\mathbb V(L)\setminus Z)\ra \mathbb R^*_+, \, v\mapsto -||v||^2$$
is strictly plurisubharmonic.
\end{definition}

\begin{remark}
A positive hermitian line bundle is called "metrically pseudoconcave" in \cite{AndreottiTomassini69}. If $X$ is smooth, then $\overline L$ is positive if and only if the curvature form $c_1(\overline L)$ is pointwise positive.
\end{remark}

\begin{definition}
Let $\cX$ be a projective arithmetic variety, and let $\Lbar$ be a hermitian line bundle on $\cX$. We say that $\Lbar$ is \emph{ample} if $\mathcal L$ is relatively ample over $\Spec\Z$, $\Lbar$ is positive on $\cX(\C)$ and for any large enough integer $n$, there exists a basis of $H^0(\cX, \mathcal L^{\otimes n})$ consisting of strictly effective sections.
\end{definition}

\begin{remark}
Departing slightly from the definition of \cite{Zhang95}, we insist on strict positivity for the hermitian metric. This makes ampleness an open condition, see e.g. Corollary \ref{corollary:ampleness-open} below. 
\end{remark}

\begin{proposition}\label{proposition:exponentially-decreasing-sections}
Let $\cX$ be a projective arithmetic variety, and let $\Lbar$ be an ample hermitian line bundle on $\cX$. Let $\Mbar=(\mathcal M, ||.||)$ be a hermitian vector bundle on $\cX$, and let $\mathcal F$ be a coherent subsheaf of $\mathcal M$. Then there exists a positive real number $\varepsilon$ such that for any large enough integer $n$, there exists a basis of $H^0(\cX, \mathcal L^{\otimes n}\otimes\mathcal F)\subset H^0(\cX, \mathcal L^{\otimes n}\otimes\mathcal M)$ consisting of sections with norm at most $e^{-n\varepsilon}$.
\end{proposition}

\begin{proof}
Since $\mathcal L$ is relatively ample, for any large enough integers $a$ and $b$, the morphism
$$H^0(\cX, \mathcal L^{\otimes a})\otimes H^0(\cX, \mathcal L^{\otimes b}\otimes\mathcal F)\ra H^0(\cX, \mathcal L^{\otimes a+b}\otimes\mathcal F)$$
is surjective. As a consequence, for any two large enough integers $a$ and $b$, and any positive integer $n$, the morphism 
$$H^0(\cX, \mathcal L^{\otimes a})^{\otimes n}\otimes H^0(\cX, \mathcal L^{\otimes b}\otimes\mathcal F)\ra H^0(\cX, \mathcal L^{\otimes an+b}\otimes\mathcal F)$$
is surjective.

Choose $a$ large enough so that the space $H^0(\cX, \mathcal L^{\otimes a})$ has a basis consisting of sections with norm at most $\alpha$ for some $\alpha<1$. Choose $b_1, \ldots, b_a$ large enough integers that form a complete residue system modulo $l$. We can assume that the maps 
$$H^0(\cX, \mathcal L^{\otimes a})^{\otimes n}\otimes H^0(\cX, \mathcal L^{\otimes b_i}\otimes\mathcal F)\ra H^0(\cX, \mathcal L^{\otimes an+b_i}\otimes\mathcal F)$$
are surjective for all positive integer $n$ and all $i$ between $1$ and $a$. Now choose bases for the spaces $H^0(\cX, \mathcal L^{\otimes b_i}\otimes\mathcal F)$ as $i$ varies between $1$ and $a$, and let $\beta$ be an upper bound for the norm of any element of these bases. Taking products of elements of these bases, we find a subspace of full rank in $H^0(\cX, \mathcal L^{\otimes an+b_i}\otimes\mathcal F)$ which has a basis whose elements have norm at most $\alpha^n\beta$. By \cite[Lemma 1.7]{Zhang92}, this implies that $H^0(\cX, \mathcal L^{\otimes an+b_i}\otimes\mathcal F)$ has basis whose elements have norm at most $r\alpha^n\beta$, where $r$ is the rank of $H^0(\cX, \mathcal L^{\otimes an+b_i}\otimes\mathcal F)$.

The theory of Hilbert polynomials shows that the rank of $H^0(\cX, \mathcal L^{an+b_i})$ is bounded above by a polynomial in $an+b_i$. Since $\alpha<1$, the number $r\alpha^n\beta$ decreases exponentially as $an+b_i$ grows, which shows the result since any integer can be written as $an+b_i$ for some $i$ and $n$.
\end{proof}

\begin{corollary}\label{corollary:ampleness-open}
Let $\cX$ be a projective arithmetic variety. Let $\Lbar$ be an ample hermitian line bundle on $\cX$ and let $\Mbar$ be a hermitian line bundle on $\cX$. Then for any large enough integer $n$, the hermitian line bundle $\Lbar^{\otimes n}\otimes \Mbar$ is ample.
\end{corollary}

\begin{proof}
For large enough $n$, the line bundle $\Lbar^{\otimes n}\otimes \Mbar$ is positive on $\mathcal X(\C)$. Indeed, by compactness, we only have to check that for any point $x$ of $\cX(\C)$, there exists a neighborhood $U$ of $x$ such that for any large enough integer $n$, $\mathcal L^{\otimes n}\otimes \mathcal M$ is positive on $U$. Pick a neighborhood $V$ of $U$ such that the restrictions of $\mathcal L$ and $\mathcal M$ to $U$ are both trivial. Let $U\subset V$ be a neighborhood of $x$, relatively compact in $V$. To check that $\mathcal L^{\otimes n}\otimes \mathcal M$ is positive on $U$, it is now enough to notice that if $\phi$ is a smooth, strictly plurisubharmonic function on $V$, and $\psi$ is any smooth function on $V$, then $n\phi+\psi$ is a strictly plurisubharmonic smooth function on the relatively compact $U$ for any large integer $n$.

Since $\mathcal L$ is relatively ample, for any large enough integer $n$, the line bundle $\mathcal L^{\otimes n}\otimes\mathcal M$ is relatively ample and the morphisms 
$$H^0(\cX, \mathcal L^{\otimes n}\otimes\mathcal M)^{\otimes m}\ra H^0(\cX, (\mathcal L^{\otimes n}\otimes\mathcal M)^{\otimes m})$$
are surjective for any positive integer $m$. 

For large enough $n$, Proposition \ref{proposition:exponentially-decreasing-sections} guarantees that there is a basis for $H^0(\cX, \mathcal L^{\otimes n}\otimes\mathcal M)$ consisting of sections with norm at most $e^{-n\varepsilon}$ for some positive number $\varepsilon$. As a consequence, we can find a subspace of full rank in $H^0(\cX, (\mathcal L^{\otimes n}\otimes\mathcal M)^{\otimes m})$ with a basis consisting of sections with norm at most $e^{-mn\varepsilon}$. By \cite[Lemma 7.1]{Zhang92}, this implies that $H^0(\cX, (\mathcal L^{\otimes n}\otimes\mathcal M)^{\otimes m})$ itself has a basis whose elements have norm at most $re^{-mn\varepsilon}$, where $r$ is the rank of $H^0(\cX, (\mathcal L^{\otimes n}\otimes\mathcal M)^{\otimes m})$. Since again $r$ is bounded above by a polynomial in $mn$, this shows the result.
\end{proof}

\begin{corollary}\label{corollary:sections-pullback}
Let $f : \cY\ra\cX$ be a morphism of projective arithmetic varieties, and let $\Lbar$ be an ample hermitian line bundle on $\cX$. Then there exists a positive real number $\varepsilon$ such that for any large enough integer $n$, there exists a basis of $H^0(\cY, f^*\mathcal L^{\otimes n})$ consisting of sections with norm bounded above by $e^{-n\varepsilon}$.
\end{corollary}

\begin{proof}
By the projection formula, for any integer $k$, we have a canonical isomorphism
$$H^0(\cY, f^*\mathcal L^{\otimes k})\simeq H^0(\cX, \mathcal L^{\otimes k}\otimes f_*\mathcal O_{\cY}).$$
Since $\mathcal L$ is relatively ample, for any two large enough integers $n$ and $k$, the map
$$H^0(\cX, \mathcal L^{\otimes n})\otimes H^0(\cX, \mathcal L^{\otimes k}\otimes f_*\mathcal O_\cY)\ra H^0(\cX, \mathcal L^{\otimes (n+k)}\otimes f_*\mathcal O_\cY)$$
is surjective, which means that the natural map
$$H^0(\cX, \mathcal L^{\otimes n})\otimes H^0(\cY, f^*\mathcal L^{\otimes k})\ra H^0(\cY, f^*\mathcal L^{\otimes (n+k)})$$
is surjective. 

Fix a large enough integer $k$ for the previous assumption to hold. By Proposition \ref{proposition:exponentially-decreasing-sections}, the space $H^0(\cX, \mathcal L^{\otimes n})$ admits a basis consisting of elements with norm decreasing exponentially with $n$, which shows that the same property holds for $H^0(\cY, f^*\mathcal L^{\otimes (n+k)})$.
\end{proof}

\begin{corollary}\label{corollary:pull-back-ample}
Let $f$ be a finite morphism between arithmetic varieties such that $f_\C$ is unramified. The pullback of an ample hermitian line bundle by f is ample.
\end{corollary}

\begin{proof}
By the previous results, we only have to show that if $f : X\ra Y$ is an unramified finite map between complex projective varieties, and if $\overline L$ is a positive hermitian line bundle on $Y$, then $f^*\overline L$ is positive. The problem being local on $X$, we may assume that $f$ is a closed immersion.

The total space of $f^*\overline L$ can be identified with a closed subvariety of the total space of $\overline L$. Since the restriction of a strictly plurisubharmonic function is strictly plurisubharmonic, this shows that $f^*\overline L$ is positive.
\end{proof}

\begin{remark}
It is not true in general that the pullback of an ample hermitian line bundle by a finite morphism is ample. Indeed, the curvature form of the pullback vanishes at any point where the morphism is ramified. The analogue of this phenomenon at non-archimedean places is not surprising: given a finite morphism $f$ between two smooth curves over $\Q_p$ with stable reduction, it is not true in general that $f$ extends to a finite morphism between the stable models over $\Z_p$ -- but it is if $f$ is étale.
\end{remark}

\bigskip

The following is the basic example of an ample hermitian line bundle.

\begin{proposition}\label{proposition:FS-ample}
Let $n$ be a positive integer, and let $\overline{\mathcal O(1)}$ be the hermitian line bundle on $\mathbb P^n_\Z$ corresponding to the line bundle $\mathcal O(1)$ endowed with the Fubini-Study metric. Then the hermitian line bundle $\overline{\mathcal O(1)}$ is ample on $\mathbb P^n_\Z$.
\end{proposition}

\begin{proof}
The line bundle $\mathcal O(1)$ is relatively ample on $\mathbb P^n_\Z$, and the Fubini-Study metric has positive curvature.

Now let $X_0^{d_0}\ldots X_n^{d_n}$ be a monomial of total degree $d>0$, seen as a section of $\mathcal O(d)$. With respect to the Fubini-Study metric, if $x$ is a point of $\C\mathbb P^n$ with homogeneous coordinates $[x_0:\ldots:x_n]$, we have
$$||X_0^{d_0}\ldots X_n^{d_n}(x)||=\frac{|x_0^{d_0}\ldots x_n^{d_n}|}{(|x_0|^2+\ldots+|x_n|^2)^{d/2}}< 1.$$
This shows that $H^0(\mathbb P^n_\Z, \mathcal O(d))$ has a basis consisting of sections of norm bounded above by $1$, and proves the result.
\end{proof}

The following follows immediately from Proposition \ref{proposition:FS-ample} and Corollary \ref{corollary:pull-back-ample}.

\begin{corollary}
Let $\cX$ be an arithmetic variety, and let $\mathcal L$ be a relatively ample line bundle on $\cX$. Then there exists a metric $||.||$ on $\mathcal L_\C$, invariant under complex conjugation, such that the hermitian line bundle $(\mathcal L, ||.||)$ is ample.
\end{corollary}

\begin{proof}
Some positive power $\mathcal L^{\otimes n}$ of $\mathcal L$ is the pullback of the line bundle $\mathcal O(1)$ on some projective space. By Proposition \ref{proposition:FS-ample} and Corollary \ref{corollary:pull-back-ample}, the pullback of the Fubini-Study metric to $\mathcal L^{\otimes n}$ gives $\mathcal L^{\otimes n}$ the structure of an ample hermitian line bundle. 

Endow $\mathcal L$ with the hermitian metric $||.||$ whose $n$-th power is the metric above. We get a hermitian line bundle $\Lbar=(\mathcal L, ||.||)$ such that $\Lbar^{\otimes n}$ is ample. It is readily checked that $\Lbar$ is ample.
\end{proof}

\bigskip

\subsection{Arithmetic Hilbert-Samuel}
We turn to the arithmetic Hilbert-Samuel theorem, giving an estimate for $h^0_{Ar}(\cX, \Lbar^{\otimes n})$, where $\Lbar$ is ample and $n$ is large. This has been proved by Gillet-Soulé in \cite[Theorem 8 and Theorem 9]{GilletSoule92} and extended by \cite[Theorem (1.4)]{Zhang95}, see also \cite{AbbesBouche95} and \cite{Bost91}. In later arguments, we will need an estimate for the error term in the arithmetic Hilbert-Samuel theorem. In the case where the generic fiber of $\cX$ is smooth, such an estimate follows from the work of Gillet-Soulé and Bismut-Vasserot. The general case does not seem to be worked out. However, for arithmetic surfaces, an argument of Vojta gives us enough information for our later needs.

We start with a proposition relating the Hilbert-Samuel function of a hermitian line bundle and its pullback under a birational morphism.

\begin{proposition}\label{proposition:HS-and-pullback}
Let $f : \cX\ra\cY$ be a birational morphism of projective arithmetic varieties, and let $\Lbar$ be an ample hermitian line bundle on $\cY$. Then there exists a positive integer $k$ and a positive real number $C$ such that for any integer $n$ and any hermitian vector bundle $\Mbar$ on $\cX$, the following equality holds:
$$h^0_{Ar}(\cX, \Lbar^{\otimes n}\otimes\Mbar)\leq h^0_{Ar}(\cY, f^*(\Lbar^{\otimes n}\otimes\Mbar))\leq h^0_{Ar}(\cX, \Lbar^{\otimes(n+k)}\otimes \Mbar(C)).$$
\end{proposition}

\begin{proof}
Pullback of sections induces an injective map
$$f^* : H^0_{Ar}(\cX, \Lbar^{\otimes n}\otimes\Mbar)\ra H^0_{Ar}(\cY, f^*(\Lbar^{\otimes n}\otimes\Mbar)),$$
which proves the first inequality. 

The coherent sheaf $\mathcal{H}om(f_*\mathcal O_\cY, \mathcal O_\cX)$ is non-zero. As a consequence, there exists a positive integer $k$ such that the sheaf
$$\mathcal{H}om(f_*\mathcal O_\cY, \mathcal O_\cX)\otimes\mathcal L^{\otimes k}=\mathcal{H}om(f_*\mathcal O_\cY, \mathcal L^{\otimes k})$$
has a nonzero global section $\phi$. Since $f$ is birational, the morphism 
$$\phi : f_*\mathcal O_\cY\ra \mathcal L^{\otimes k}$$
is injective. If $U$ is an open subset of the compact complex manifold $\cX(\C)$ and $n$ is an integer, we endow the spaces $(\mathcal L^{\otimes n}\otimes\mathcal M\otimes f_*\mathcal O_\cY)(U)=(f^*\mathcal L^{\otimes n}\otimes\mathcal M)(f^{-1}(U))$ and $(\mathcal L^{\otimes(n+k)}\otimes\mathcal M)(U)$ with the sup norm. Since $\cX(\C)$ is compact, we can find a constant $C$ such that the maps
$$\phi_U : f_*\mathcal O_\cY(U)\ra \mathcal L^{\otimes k}(U)$$
all have norm bounded above by $e^C$. As a consequence, all the maps
$$\mathrm{Id}\otimes\phi_U : (\mathcal L^{\otimes n}\otimes\mathcal M\otimes f_*\mathcal O_\cY)(U)\ra (\mathcal L^{\otimes(n+k)}\otimes\mathcal M)(U)$$
have norm bounded above by $e^C$ as well, and the induced map
$$H^0(\cY, f^*(\mathcal L^{\otimes n}\otimes\mathcal M))\ra H^0(\cX, \mathcal L^{\otimes (n+k)}\otimes\mathcal M)$$
has norm bounded above by $e^C$. Since this map is injective, we have an injection 
$$H^0_{Ar}(\cY, f^*(\Lbar^{\otimes n}\otimes\Mbar))\ra H^0_{Ar}(\cX, \Lbar^{\otimes (n+k)}\otimes\Mbar(C)),$$
which shows the second inequality.
\end{proof}

\begin{theorem}\label{theorem:arithmetic-Hilbert-Samuel}
Let $\cX$ be a projective arithmetic variety of dimension $d$, let $\Lbar$ be an ample hermitian line bundle on $\cX$, and let $\Mbar$ be a hermitian vector bundle of rank $r$. 
\begin{enumerate}[(i)]
\item As $n$ tends to $\infty$, we have
$$h^0_{Ar}(\cX, \Lbar^{\otimes n}\otimes\Mbar)= \frac{r}{d!}\Lbar^d n^d+o(n^d);$$

\item  if $\cX_\Q$ is smooth, then 
$$h^0_{Ar}(\cX, \Lbar^{\otimes n}\otimes\Mbar)= \frac{r}{d!}\Lbar^d n^d+O(n^{d-1}\log n)$$
as $n$ tends to $\infty$;

\item if $d= 2$, then 
$$h^0_{Ar}(\cX, \Lbar^{\otimes n}\otimes\Mbar)\geq \frac{r}{2}\Lbar^2 n^2+O(n\log n)$$
as $n$ tends to $\infty$.
\end{enumerate}
\end{theorem}

\begin{proof}
The first statement can be found in \cite[Corollary 2.7(1)]{Yuan08}. It is a consequence of the extension by Zhang in \cite[Theorem (1.4)]{Zhang95} of the arithmetic Hilbert-Samuel theorem  of Gillet-Soulé of \cite[Theorem 8]{GilletSoule92}, together with \cite[Theorem 1]{GilletSoule91}.

\bigskip

Let us prove the second statement. Choose a Kähler metric on $\cX(\C)$, assumed to be invariant under complex conjugation, and write $\chi_{L^2}(\Lbar^{\otimes n}\otimes\Mbar)$ (resp. $\chi_{sup}(\Lbar^{\otimes n}\otimes\Mbar)$) for the logarithm of the covolume of $H^0(\cX, \Lbar^{\otimes n}\otimes\Mbar)$ for the associated $L^2$ norm (resp. for the sup norm). Then by \cite[Theorem 8]{GilletSoule92}, we have
$$\chi_{L^2}(\Lbar^{\otimes n}\otimes\Mbar)=\frac{r}{d!}\Lbar^d n^d+O(n^{d-1}\log n).$$
By the Gromov inequality as in for instance \cite[Corollary 2.7(2)]{Yuan08}, this implies
$$\chi_{sup}(\Lbar^{\otimes n}\otimes\Mbar)=\frac{r}{d!}\Lbar^d n^d+O(n^{d-1}\log n).$$

Consider the lattice $\Lambda=H^0(\cX, \Lbar^{\otimes n}\otimes\Mbar)$, endowed with the sup norm. Then since $\Lambda$ is generated by elements of norm strictly smaller than $1$, the dual of $\Lambda$ does not contain any nonzero element of norm smaller than $1$. Furthermore, the geometric version of the Hilbert-Samuel theorem shows that the rank of $\Lambda$ has the form $O(n^{d-1})$. By \cite[Theorem 1]{GilletSoule91}, we get 
$$|h^0_{Ar}(\cX, \Lbar^{\otimes n}\otimes\Mbar)-\chi_{sup}(\Lbar^{\otimes n}\otimes\Mbar)|=O(n^{d-1}\log n),$$
which proves the desired result.

\bigskip

We now prove the last statement. Let $f : \cY\ra \cX$ be the normalization of $\cX$, so that $f$ is birational, finite, and $\cY$ has smooth generic fiber. 

Since $f$ is finite, the line bundle $f^*\mathcal L$ is ample. Let $\Lbar'$ be $f^*\Lbar$ and $\Mbar'$ be $f^*\Mbar$. Choose a Kähler metric on $\cY(\C)$, assumed to be invariant under complex conjugation, and again write $\chi_{L^2}(\Lbar'^{\otimes n}\otimes\Mbar')$ for the logarithm of the covolume of $H^0(\cY, \Lbar'^{\otimes n}\otimes\Mbar')$ for the associated $L^2$ norm.

By Proposition \ref{proposition:HS-and-pullback}, we can find a constant $C$ and an integer $k$ such that for any integer $n$ greater or equal to $k$, we have 
\begin{equation}\label{equation:comparison-pullback}
h^0_{Ar}(\cX, \Lbar^{\otimes n}\otimes\Mbar)\geq h^0_{Ar}(\cY, \Lbar'^{\otimes n-k}\otimes\Mbar'(-C)).
\end{equation}

Applying the arithmetic Riemann-Roch theorem for $n$ large enough so that the higher cohomology groups of $\mathcal L'^{\otimes n}\otimes\mathcal M'$ vanish, we get the following equality:
$$\chi_{L^2}(\Lbar'^{\otimes n}\otimes\Mbar'(-C))-\frac{1}{2}T_n=\frac{r}{2}\Lbar^2 n^2+O(n),$$
where by $T_n$ we denote the analytic torsion of the hermitian vector bundle $\Lbar'^{\otimes n}\otimes\Mbar'$. The equality above is proved via the computations of \cite[Theorem 8]{GilletSoule92}, or \cite[Théorème 1]{GilletSoule88}. In contrast with the usual setting of Hilbert-Samuel, note that the curvature form of $\Lbar'$ might not be positive everywhere, so that we cannot apply the estimates of \cite{BismutVasserot89} for $T_n$. However, since the dimension of $\cX$ is $2$, we have 
$$T_n=\zeta_{1,n}'(0),$$
where $\zeta_1$ is the zeta function of the Laplace operator acting on forms of type $(0,1)$ with values in $\Lbar'^{\otimes n}\otimes\Mbar'(-C)$. We can use an estimate of Vojta to control the analytic torsion $T_n$. Indeed, by \cite[Proposition 2.7.6]{Vojta91}, we have
$$\zeta'_{1,n}(0)\geq -Kn\log n$$
for some constant $K$, so that 
\begin{equation}\label{equation:estimate-degree}
\chi_{L^2}(\Lbar'^{\otimes n}\otimes\Mbar'(-C))=\frac{r}{2}\Lbar^2 n^2+\frac{1}{2}T_n+O(n)\geq \frac{r}{2}\Lbar^2 n^2+O(n\log n).
\end{equation}

Combining as above Gromov's inequality, \cite[Theorem 1]{GilletSoule91}, Corollary \ref{corollary:sections-pullback} and the geometric version of Hilbert-Samuel, we can write
$$|h^0_{Ar}(\cY, \Lbar'^{\otimes n}\otimes\Mbar'(-C))-\chi_{L^2}(\Lbar'^{\otimes n}\otimes\Mbar'(-C))|=O(n\log n),$$
which together with (\ref{equation:estimate-degree}) gives the estimate
$$h^0_{Ar}(\cY, \Lbar'^{\otimes n}\otimes\Mbar'(-C))\geq \frac{r}{2}\Lbar^2 n^2+O(n\log n).$$

From (\ref{equation:comparison-pullback}), we finally obtain
$$h^0_{Ar}(\cX, \Lbar^{\otimes n}\otimes\Mbar)\geq \frac{r}{2}\Lbar^2 (n-k)^2+O(n\log n) \geq \frac{r}{2}\Lbar^2 n^2+O(n\log n).$$

\end{proof}


\subsection{Restriction of sections}\label{subsection:restriction}

Let $k$ be a field, and let $X$ be a projective variety over $k$. Let $\mathcal L$ be an ample line bundle on $X$. If $Y$ is any closed subscheme of $X$, consider the restriction maps
$$\phi_n : H^0(X, \mathcal L^{\otimes n})\ra H^0(Y, \mathcal L_{|Y}^{\otimes n}).$$
The map $\phi_n$ is surjective if $n$ is large enough and, obviously, there are bijections between the different fibers of $\phi_n$ when it is surjective. 

In this section, we give arithmetic analogues of these results, looking at $H^0_{Ar}$ instead of $H^0$ -- this is Theorem \ref{theorem:uniform-surjectivity}. Furthermore, we show in Theorem \ref{theorem:independent-restriction} that the lower bound on the dimension of the image of the restriction map can be given to be independent of $Y$. 

\bigskip

In the following, let $\cX$ be a projective arithmetic variety, and let $\Lbar$ be an ample hermitian line bundle on $\cX$. If $n$ is an integer, we denote by $\Lambda_n$ the space $H^0(\cX, \Lbar^{\otimes n})$ endowed with the sup norm, and we write $r_n$ for its rank. If $r$ is a nonnegative real number, let $B_n(r)$ be the unit ball in $\Lambda_n$, and let $B_n(r)_\R$ be the unit ball in $\Lambda_n\otimes\R$. Let $\mathrm{Vol}$ denote the volume with respect to the unique invariant measure on $\Lambda_n\otimes\R$ for which $\mathrm{Vol}(B_n(1)_\R)=1$.

If $\mathcal I$ is a quasi-coherent sheaf of ideals on $\cX$, we write $\Lambda_n^{\mathcal I}$ for the subspace $H^0(\cX, \mathcal L^{\otimes n}\otimes \mathcal I)$ of $H^0(\cX, \mathcal L^{\otimes n})$, endowed with the induced norm. We write $r_n^{\mathcal I}$, $B_n(r)^{\mathcal I}$, $B_n(r)_\R^{\mathcal I}$, $\mathrm{Vol}^{\mathcal I}$ for the corresponding objects. 

We gather a few results regarding point counting in the lattices $\Lambda_n^{\mathcal I}$.

\begin{lemma}\label{lemma:size-of-balls}
Let $\alpha$ and $\eta$ be positive real numbers, with $0<\alpha<1$. Let $C$ be any real number. Then, as $n$ goes to infinity, we have, for any positive $r>e^{-n^\alpha}$,
$$\mathrm{Vol}^{\mathcal I}(B_n(r+Ce^{-n\eta})_\R^{\mathcal I})=r^{r_n^{\mathcal I}}(1+Cr^{-1}r_n^{\mathcal I}e^{-n\eta}+o(r^{-1}r_n^{\mathcal I}e^{-n\eta})),$$
where the implied constants depend on $\alpha, C$ and $\eta$, but not on $r$.
\end{lemma}

\begin{proof}
By Riemann-Roch, $r_n^{\mathcal I}$ grows at most polynomially in $n$. As a consequence, $r^{-1}r_n^{\mathcal I}e^{-n\eta}$ goes to $0$ as $n$ goes to infinity, and we can write
\begin{align*}
\mathrm{Vol}^{\mathcal I}(B_n(r+Ce^{-n\eta})_\R^{\mathcal I}) &=(r+Ce^{-n\eta})^{r_n^{\mathcal I}}\\
&=r^{r_n^{\mathcal I}}\exp(r_n^{\mathcal I}\log(1+Cr^{-1}e^{-n\eta}))\\
&=r^{r_n^{\mathcal I}}\exp(Cr^{-1}r_n^{\mathcal I}e^{-n\eta}+o(r^{-1}r_n^{\mathcal I}e^{-n\eta}))\\
&=r^{r_n^{\mathcal I}}(1+Cr^{-1}r_n^{\mathcal I}e^{-n\eta}+o(r^{-1}r_n^{\mathcal I}e^{-n\eta})).
\end{align*}
\end{proof}

Fix $\mathcal I$ as above. Let $n$ be a large enough integer. By Proposition \ref{proposition:exponentially-decreasing-sections}, we can find a positive number $\varepsilon^{\mathcal I}$, independent of $n$, and a basis $\sigma_1, \ldots, \sigma_{r_n^{\mathcal I}}$ of $\Lambda_n^{\mathcal I}$ such that $||\sigma_i||\leq e^{-n\varepsilon^{\mathcal I}}$ for all $i\in\{1, \ldots, r_n^{\mathcal I}\}$. Write 
\begin{equation}\label{equation:fundamental-domain}
D_n^{\mathcal I}=\Big\{\sum_{i=1}^{r_n^{\mathcal I}}\lambda_i\sigma_i | \forall i\in\{1, \ldots, n\},\,0\leq \lambda_i<1\Big\}.
\end{equation}

\begin{proposition}\label{proposition:number-of-points-ball}
Let $\alpha$ be a positive number with $0<\alpha<1$. As $n$ tends to $\infty$, we have, for any $r>e^{-n^\alpha}$,
$$|B_n(r)^{\mathcal I}|\mathrm{Vol}^{\mathcal I}(D_n^{\mathcal I})\sim r^{r_n^{\mathcal I}},$$
where the implied constants depend on $\alpha$ but not on $r$.
\end{proposition}

\begin{proof}
Let $n$ be a large enough integer. As $\sigma$ runs through the elements of $\Lambda_n^{\mathcal I}$, the sets $\sigma+D_n^{\mathcal I}$ are pairwise disjoint and cover $\Lambda_n^{\mathcal I}\otimes\R$. Furthermore, the diameter of $D_n^{\mathcal I}$ is bounded above by $r_n^{\mathcal I}e^{-n\varepsilon^{\mathcal I}}$. As a consequence, if $\sigma$ is any element of $\Lambda_n^{\mathcal I}$, then 
$$\sigma+D_n^{\mathcal I}\subset B_n(||\sigma||+r_n^{\mathcal I}e^{-n\varepsilon^{\mathcal I}})_\R^{\mathcal I}$$
and 
$$(\sigma+D_n^{\mathcal I})\cap B_n(||\sigma||-r_ne^{-n\varepsilon^{\mathcal I}})_\R^{\mathcal I}=\emptyset.$$

As a consequence, we have
$$\mathrm{Vol}^{\mathcal I}(B_n(r-r_ne^{-n\varepsilon^{\mathcal I}})_\R^{\mathcal I})\leq |B_n(r)^{\mathcal I}|\mathrm{Vol}^{\mathcal I}(D_n^{\mathcal I})\leq \mathrm{Vol}^{\mathcal I}(r+r_n^{\mathcal I}e^{-n\varepsilon^{\mathcal I}})_\R^{\mathcal I}.$$
Since $r_n^{\mathcal I}$ grows at most polynomially in $n$, Lemma \ref{lemma:size-of-balls} shows that both the left and right terms are equivalent to $r^{r_n^{\mathcal I}}$ as $n$ goes to infinity.
\end{proof}

\begin{proposition}\label{proposition:points-in-annulus}
Let $\alpha$ and $\eta$ be positive real numbers with $0<\alpha<1$. Let $C$ be any real number. Then, as $n$ tends to $\infty$, there exists a positive real number $\eta'$ such that we have, for any positive $r$ bounded below by $e^{-n^\alpha}$,
$$\frac{||B_n(r+Ce^{-n\eta})^{\mathcal I}|-|B_n(r)^{\mathcal I}||}{|B_n(r)^{\mathcal I}|}=O(e^{-n\eta'}),$$
where the implied constants depend on $\alpha, C$ and $\eta$.
\end{proposition}

\begin{proof}
We assume that $C$ is positive. The case where $C$ is negative (or zero) can be treated by the same computations.

Let $\eta'$ be a positive number strictly smaller than both $\varepsilon^{\mathcal I}$ and $\eta$. Since the $\sigma+D_n^{\mathcal I}$ are pairwise disjoint as $\sigma$ runs through the elements of $\Lambda_n^{\mathcal I}$, we get, for large enough $n$
\begin{align*}
\big(|B_n(r+Ce^{-n\eta})^{\mathcal I}|-|B_n(r)^{\mathcal I}|\big)\mathrm{Vol}^{\mathcal I}(D_n^{\mathcal I}) & \leq \mathrm{Vol}^{\mathcal I}(B_n(r+Ce^{-n\eta}+r_n^{\mathcal I}e^{-n\varepsilon^{\mathcal I}})_\R^{\mathcal I})-\mathrm{Vol}^{\mathcal I}(B_n(r-r_n^{\mathcal I}e^{-n\varepsilon^{\mathcal I}})_\R^{\mathcal I})\\
&\leq \mathrm{Vol}^{\mathcal I}(B_n(r+e^{-n\eta'})_\R^{\mathcal I})-\mathrm{Vol}^{\mathcal I}(B_n(r-e^{-n\eta'})_\R^{\mathcal I})\\
& \sim 2r^{r_n^{\mathcal I}-1}r_n^{\mathcal I}e^{-n\eta'},\\
\end{align*}
where in the last line we applied Lemma \ref{lemma:size-of-balls}.

Putting the previous estimate together with Proposition \ref{proposition:number-of-points-ball} and replacing $\eta'$ by a smaller positive number, we get the desired result.
\end{proof}

\bigskip

The following is a first step in controlling restriction maps. 

\begin{proposition}\label{proposition:varying-quotient}
Let $\alpha$ be a positive number with $0<\alpha<1$. There exists a positive constant $\eta$ such that for any large enough integer $n$, if $N$ is any positive integer bounded above by $e^{n^\alpha}$, then the following holds:
\begin{enumerate}[(i)]
\item the map $\phi_{n,N} : \Lambda_n\ra\Lambda_n/N\Lambda_N$ is surjective;
\item for any two $s, s'$in $\Lambda_n/N\Lambda_n$, we have 
$$\frac{\big||\phi_{n, N}^{-1}(s)|-|\phi_{n, N}^{-1}(s')|\big|}{|\phi_{n, N}^{-1}(s)|}\leq e^{-n\eta}.$$
\end{enumerate}
\end{proposition}

\begin{proof}
Let $r_n$ be the rank of $\Lambda_n$. Let $n$ be a positive integer, which will be chosen large enough, and let $N$ be an integer bounded above by $n^\alpha$.

By Proposition \ref{proposition:exponentially-decreasing-sections}, we can find a positive number $\varepsilon$, independent of $n$, and a basis $\sigma_1, \ldots, \sigma_{r_n}$ of $\Lambda_n$ such that $||\sigma_i||\leq e^{-n\varepsilon}$ for all $i\in\{1, \ldots, r_n\}$. Now any element of $\Lambda_n/N\Lambda_n$ is the restriction of an element of $\Lambda_n$ of the form 
$$\sigma=\lambda_1\sigma_1+\ldots + \lambda_{r_n}\sigma_{r_n},$$
where the $\lambda_i$ are integers between $0$ and $N-1$. We have
$$||\sigma||< Nr_ne^{-n\varepsilon}\leq n^\alpha r_ne^{-n\varepsilon}.$$
We know that $r_n$ is a polynomial in $n$ for large enough $n$, so that any $\sigma$ as above has norm at most $1$ for large enough $n$. This shows that the map $\phi_{n, N}$ is surjective and proves $(i)$.

\bigskip

We now proceed to the proof of $(ii)$. Let $n$ be a large enough integer. By the discussion above, we can find a positive real number $\varepsilon'$ such that for any large enough integer $n$, and any $s$ in $\Lambda_n/N\Lambda_n$, there exists an element $\sigma_0$ in $\Lambda_n$ with $||\sigma_0||\leq e^{-n\varepsilon'}$ that restricts to $s$. We have 
$$\phi_{n, N}^{-1}(s)=\{\sigma_0+N\sigma|\,\sigma\in \Lambda_n), ||\sigma_0+N\sigma||<1\},$$
so that, up to replacing $\varepsilon'$ by a smaller positive number
$$|B_n(1/N-e^{-n\varepsilon'})|\leq |\phi_{n, N}^{-1}(s)|\leq |B_n(1/N+e^{-n\varepsilon'})|$$
and
\begin{equation}\label{equation:estimate-phi}
\big|\,|\phi_{n, N}^{-1}(s)|-|\phi_{n, N}^{-1}(s')|\,\big|\leq |B_n(1/N+e^{-n\varepsilon'})|-|B_n(1/N-e^{-n\varepsilon'})|
\end{equation}
for any two $s, s'$ in $\Lambda_n/N\Lambda_n$. We conclude by applying Proposition \ref{proposition:points-in-annulus}.

\end{proof}

\bigskip 

The following is a key property of ample line bundles.

\begin{theorem}\label{theorem:uniform-surjectivity}
Let $\cX$ be a projective arithmetic variety, and let $\Lbar$ be an ample line bundle on $\cX$. 
Let $\cY$ be a closed subscheme of $\cX$, such that $\cY_\Q$ is reduced. If $n$ is a positive integer, let 
$$\phi_n : H^0(\cX, \mathcal L^{\otimes n})\ra H^0(\cY, \mathcal L^{\otimes n})$$
be the restriction map. 
For any positive $\varepsilon$, define 
$$\Lambda_n^{\varepsilon}=H^0_{Ar}(\cX, \Lbar^{\otimes n})\cap \phi_n^{-1}(H^0_{Ar}(\cY, \Lbar(-\varepsilon)^{\otimes n})).$$
Write $\psi_n:=(\phi_n)_{|\Lambda_n^\varepsilon}.$ Then the following holds:
\begin{enumerate}[(i)]
\item For any large enough integer $n$, the restriction map
$$\psi_n : \Lambda_n^\varepsilon \ra H^0_{Ar}(\cY, \Lbar(-\varepsilon)^{\otimes n})$$
is surjective;
\item there exists a positive constant $\eta$ such that for any large enough integer $n$, and any two $s, s'$ in $H^0_{Ar}(\cY, \Lbar(-\varepsilon)^{\otimes n})$, we have 
$$\frac{\big||\phi_n^{-1}(s)|-|\phi_n^{-1}(s')|\big|}{|\phi_n^{-1}(s)|}\leq e^{-n\eta};$$
\item for any $s\in H^0_{Ar}(\cY, \Lbar(-\varepsilon)^{\otimes n})$, we have
$$\Big|\phi_n^{-1}(s)-\frac{|\Lambda_n^{\varepsilon}|}{|H^0_{Ar}(\cY, \Lbar(-\varepsilon)^{\otimes n})|} \Big| \leq e^{-n\eta}|\phi_n^{-1}(s)|.$$
\end{enumerate}
\end{theorem}

\begin{proof}
Fix $\varepsilon>0$. The group
$$\{\sigma\in H^0_{Ar}(\cY, \Lbar(-\varepsilon)^{\otimes n}), ||\sigma_\C||=0\}$$
is the torsion subgroup of $H^0_{Ar}(\cY, \Lbar(-\varepsilon)^{\otimes n})$, which we denote by $H^0_{Ar}(\cY, \mathcal L^{\otimes n})^{0}$ -- note that this group does not depend on $\varepsilon$ nor the hermitian metric. Let $N$ be a positive integer with
$$NH^0_{Ar}(\cY, \mathcal L^{\otimes n})^{0}=0.$$

Assume $n$ is large enough. The restriction map 
$$\phi_n : \Lambda_n = H^0(\cX, \mathcal L^{\otimes n})\ra H^0(\cY, \mathcal L^{\otimes n})$$ 
is surjective since $\mathcal L$ is relatively ample. Then the map 
$$\Lambda_n/N\Lambda_n\ra H^0_{Ar}(\cY, \mathcal L^{\otimes n})^{0}$$
is well-defined and surjective as well. Applying Proposition \ref{proposition:varying-quotient}, this shows that the image of $\psi_n$ contains $H^0_{Ar}(\cY, \mathcal L^{\otimes n})^{0}$.

\bigskip

Let $s$ be an element of $H^0_{Ar}(\cY, \Lbar(-\varepsilon)^{\otimes n})$. Let $\varepsilon'$ be a real number with $0<\varepsilon'<\varepsilon.$ Apply\footnote{The only assumption necessary in \cite{Bost04} is that $\cY_\C$ and $\cX_\C$ are reduced, see \cite{Randriam06} for a statement that makes this explicit.} \cite[Theorem A.1]{Bost04} to $\cY_\C\hookrightarrow \cX_\C$. If $n$ is large, we can find a section $\sigma$ of $\mathcal L^{\otimes n}$ on $\cX_\C$ with $||\sigma||\leq e^{-n\varepsilon'}$ and $\sigma_{|\cY_\C}=s_\C$. Up to replacing $\sigma$ with $\sigma+\overline\sigma$, and applying the argument to $\varepsilon'\leq \varepsilon$, we may assume that $\sigma$ is a section of $\mathcal L^{\otimes n}$ over $\cX_\R$, that is, 
$$\sigma\in B_n(e^{-n\varepsilon'})_\R.$$

Let $\mathcal I$ be the ideal of $\cY$ in $\cX$. The kernel of the -- surjective when $n$ is large enough -- restriction map 
$$\phi_n : \Lambda_n\ra H^0(\cY, \mathcal L^{\otimes n})$$
is $\Lambda_n^{\mathcal I}$. Let $\sigma'$ be an element of $\Lambda_n$ mapping to $s$.  Then $\sigma\in(\Lambda_n^{\mathcal I})_\R+\sigma'$.

The fundamental domain $D_n^{\mathcal I}$ defined in (\ref{equation:fundamental-domain}) has diameter bounded above by $r_ne^{-n\varepsilon^{\mathcal I}}.$ In particular, we can find $\sigma''\in \Lambda_n^{\mathcal I}+\sigma'$ with 
$$||\sigma''-\sigma||\leq r_ne^{-n\varepsilon^{\mathcal I}},$$
so that 
$$||\sigma''||\leq e^{-n\varepsilon'}+r_ne^{-n\varepsilon^{\mathcal I}}<1$$
for large enough $n$. We have $\psi_n(\sigma)_\C=s_\C$, i.e., $\psi_n(\sigma)-\sigma$ is torsion. This shows that the image of $\psi_n$ maps surjectively onto the quotient of $H^0_{Ar}(\cY, \Lbar(-\varepsilon)^{\otimes n})$ by $H^0_{Ar}(\cY, \mathcal L^{\otimes n})^{0}$. Since we showed above that it contains $H^0_{Ar}(\cY, \mathcal L^{\otimes n})^{0}$, this proves that $\psi_n$ is surjective.

\bigskip

Apply statement $(i)$ after replacing $\Lbar$ with $\Lbar(-\delta)$, where $\delta>0$ is chosen small enough so that $\Lbar(-\delta)$ is ample. Then if $\varepsilon>\delta$ and $n$ is large enough, for any $s\in H^0_{Ar}(\cY, \Lbar(-\varepsilon)^{\otimes n})$, we can find $\sigma_0\in H^0_{Ar}(\cX, \Lbar(-\delta)^{\otimes n})$ that restricts to $s$.

To prove $(ii)$, we argue as in Proposition \ref{proposition:varying-quotient}. Let $s$ and $\sigma_0$ be as above. Then 
$$\psi_n^{-1}(s)=\{\sigma_0+\sigma | \sigma\in \Lambda_n^{\mathcal I}, ||\sigma_0+\sigma||<1\}$$
and 
$$|B_n(1-e^{-n\delta})^{\mathcal I}|\leq |\psi_n^{-1}(s)|\leq B_n(1+e^{-n\delta})^{\mathcal I}.$$
Using Proposition \ref{proposition:points-in-annulus} again, this proves $(ii)$. 

\bigskip

To prove $(iii)$, write 
$$|\Lambda_n^\varepsilon|=\sum_{s\in H^0_{Ar}(\cY, \Lbar(-\varepsilon^{\otimes n})} |\psi_n^{-1}(s)|,$$
so that for any large enough $n$ and any $s\in H^0_{Ar}(\cY, \Lbar(-\varepsilon)^{\otimes n})$, we have 
$$\big| |\Lambda_n^\varepsilon|-|\psi_n^{-1}(s)|\,|H^0_{Ar}(\cY, \Lbar(-\varepsilon)^{\otimes n})|\big|\leq e^{-n\eta}|\psi_n^{-1}(s)| |H^0_{Ar}(\cY, \Lbar(-\varepsilon)^{\otimes n})|.$$
\end{proof}

We keep the notations of the theorem. 

\begin{corollary}\label{corollary:density-and-restriction}
Let $E$ be a subset of $\bigcup_{n>0} H^0_{Ar}(\cY, \Lbar(-\varepsilon)^{\otimes n})$. Set
$$E':=\{\sigma\in \bigcup_{n>0} \Lambda_n^{\varepsilon}, \sigma_{|\cY}\in E\}.$$

For any $0\leq\rho\leq 1$, the set $E$ has density $\rho$ if and only if $E'$ has density $\rho$.
\end{corollary}

\begin{proof}
For any positive integer $n$, define
$$E_n:=E\cap H^0_{Ar}(\cY, \Lbar(-\varepsilon)^{\otimes n}), \, E'_n:=E'\cap H^0_{Ar}(\cX, \Lbar^{\otimes n}).$$
Denoting by $\psi_n$ the restriction maps as before, we can write 
$$|E'_n|=\sum_{s\in E_n} \psi_n^{-1}(s).$$
By Theorem \ref{theorem:uniform-surjectivity}, we can find a positive constant $\eta$ such that, for large enough $n$,
$$\Big||E'_n|-\frac{|E_n|}{|H^0_{Ar}(\cY, \Lbar(-\varepsilon)^{\otimes n})|}|H^0_{Ar}(\cX, \Lbar^{\otimes n})| \Big| \leq e^{-n\eta}|E'_n|$$
and
$$\Big| \frac{|E'_n|}{|H^0_{Ar}(\cX, \Lbar^{\otimes n})|}-\frac{|E_n|}{|H^0_{Ar}(\cY, \Lbar(-\varepsilon)^{\otimes n})|}\Big| \leq e^{-n\eta}.$$
Letting $n$ tend to $\infty$ gives us the result we were looking for.
\end{proof}

As a special case of the theorem, we get the following.

\begin{corollary}\label{corollary:uniform-surjectivity-p}
Let $\cX$ be a projective arithmetic variety, and let $\Lbar$ be an ample line bundle on $\cX$. Let $Y$ be a closed subscheme of $\cX$ lying over $\Z/N\Z$. Then for any large enough integer $n$, the restriction map 
$$\phi_n : H^0_{Ar}(\cX, \Lbar^{\otimes n})\ra H^0(Y, \mathcal L^{\otimes n})$$
is surjective and there exists a positive constant $\eta$ such that for any $s\in H^0(Y, \mathcal L^{\otimes n})$, we have
$$\Big|\phi_n^{-1}(s)-\frac{|H^0_{Ar}(\cX, \Lbar^{\otimes n})|}{|H^0(Y, \mathcal L^{\otimes n})|} \Big| \leq e^{-n\eta}|\phi_n^{-1}(s)|.$$
\end{corollary}

\bigskip

We now turn to uniform lower bounds on the image of restriction maps. We first deal with a geometric result.

\begin{proposition}\label{proposition:uniform-lower-bound-geometric}
Let $S$ be a noetherian scheme, and let $X$ be a projective scheme over $S$. Let $\mathcal L$ be a line bundle on $X$, relatively ample over $S$. Then there exists an integer $N$ and a positive constant $C$ such that for any point $s$ of $S$, any closed subscheme $Y$ of $X_s$ of positive dimension $d$, and any $n\geq N$, the image of the restriction map 
$$H^0(X_s, \mathcal L^{\otimes n})\ra H^0(Y, \mathcal L^{\otimes n})$$
has dimension at least $Cn^d$.
\end{proposition}

\begin{proof}
Since $S$ is noetherian, we can find an integer $N$ such that for any point $s$ of $S$ and any integer $n\geq N$, the restriction of $\mathcal L^{\otimes n}$ to $X_s$ is very ample. 

Let $s$ be a point of $S$, and let $Y$ be a positive-dimensional closed subscheme of $X_s$. Let $k$ be an infinite field containing the residue field of $s$, and write $X_k$ for the base change of $X_s$ to $k$. 

Since $\mathcal L^{\otimes N}$ is very ample on $X_k$ and $k$ is infinite, we can find a $d+1$-dimensional subspace $V\subset H^0(X, \mathcal L^{\otimes n})$ such that the restriction to $Y$ of the rational map
$$\phi : X\dashrightarrow \mathbb{P}(V^*)$$
is dominant. Let $\sigma_0, \ldots, \sigma_d$ be a basis of $V$, and let $H_0$ be the divisor $\mathrm{div}(\sigma_0)$. Identify the subspace of $\mathbb{P}(V^*)$ defined by $\sigma_0\neq 0$ to the standard affine space $\mathbb A^d_k$ with coordinates $x_1, \ldots, x_d$. Then the map $\phi$ is defined outside $H_0$ -- as certainly the base locus of $V$ is contained in $H_0$, and maps onto $\mathbb A^d_k$.

For any positive integer $r$ and any integer $n\geq (r+1)N$, the line bundle $\mathcal L^{\otimes n}(-rH_0)\simeq\mathcal L^{\otimes n-rN}$ is very ample. In particular, we can find a section $\sigma$ of $\mathcal L^{\otimes n}$ that vanishes to the order $r$ along $H_\infty$, but does not vanish on $Y$. 

Let $P\in k[x_1, \ldots, x_d]$ be a polynomial of degree at most $r$, considered as a morphism $\mathbb A^d_k\ra\mathbb A^1_k$. Since $\sigma$ vanishes to the order $r$ along $H_\infty$, the section $(P\circ\phi) \sigma$ of $\mathcal L^{\otimes n}$, which is a priori defined only outside $H_0$, defines a global section of $\mathcal L^{\otimes n}$. Because $\sigma$ does not vanish on $Y$, the restrictions $(P\circ\phi)\sigma_{|Y}$ are linearly independent as sections of $L^{\otimes n}_{|Y}$ as $P$ varies. In particular, the image of the restriction map 
$$H^0(X, \mathcal L^{\otimes n})\ra H^0(Y, \mathcal L^{\otimes n})$$
has dimension at least equal to the dimension of the space of polynomials of degree at most $r$ in $x_1, \ldots, x_d$, so that it has dimension at least 
$${{r+d}\choose{d}}=\frac{1}{d!}r^d+O(r^{d-1})$$
for any $r$ with $r+1\leq n/N$. This proves the result.
\end{proof}

\begin{theorem}\label{theorem:independent-restriction}
Let $\cX$ be a projective arithmetic variety, and let $\Lbar$ be an ample hermitian line bundle on $\cX$. If $\cY$ is a subscheme of $\cX$, let 
$$\phi_{n, \cY} : H^0(\cX, \mathcal L^{\otimes n})\ra H^0(Y, \mathcal L^{\otimes n})$$
be the restriction map.

Then there exists an integer $N$ and a positive real number $\eta$ such that for any positive-dimensional closed subscheme $\cY$ of $\cX$,
we have 
$$\frac{|\Ker(\phi_{n, \cY})\cap H^0_{Ar}(\cX, \Lbar^{\otimes n})|}{|H^0_{Ar}(\cX, \Lbar^{\otimes n})|}=O(e^{-n^d\eta}),$$
where the implied constant depends on $\cX$ and $\Lbar$, but not on $\cY$.
\end{theorem}

\begin{proof}
We only have to consider those $\cY$ that are irreducible. Let us first assume that $\cY$ is flat over $\Spec\Z$. By Proposition \ref{proposition:uniform-lower-bound-geometric} applied to $\cX_\Q$ and $\cY_\Q$, there exists a positive constant $C$, independent of $\cY$, such that if $n$ is larger than some integer $N$, independent of $\cY$, then the kernel of the restriction map 
$$\Lambda_n=H^0(\cX, \mathcal L^{\otimes n})\ra H^0(\cY, \mathcal L^{\otimes n})$$
has corank at least $Cn^{d-1}$ for large enough $n$. Let $H_n$ denote this kernel.
Let $k_n$ be the corank of $H_n$. For $n\geq N$, we have
\begin{equation}\label{equation:bound-corank}
k_n\geq Cn^{d-1}.
\end{equation}

Up to enlarging $N$, Proposition \ref{proposition:exponentially-decreasing-sections} allows us to assume that for $n\geq N$, $\Lambda_n$ has a basis consisting of elements with norm at most $e^{-n\varepsilon}$ for some $\varepsilon>0$. For $n\geq N$, we can
find elements $\sigma_1, \ldots, \sigma_{k_n}$ of $\Lambda_n$ that are linearly independent in $\Lambda_n/H_n$ and satisfy $||\sigma_i||\leq e^{-n\varepsilon}$ for $i\in\{1, \ldots, k_n\}$.

Let $\eta$ be a positive number smaller than $\varepsilon$. Then for any $\sigma\in H_n\cap B_n(1)_\R$ and any integers $\lambda_1, \ldots, \lambda_{k_n}$ with $|\lambda_i|\leq e^{n\eta}$ for $i\in\{1, \ldots, k_n\}$, we have 
$$||\sigma+\sum_{i=1}^{k_n}\lambda_i\sigma_i||\leq 1+k_ne^{-n(\varepsilon-\eta)}.$$
Furthermore, as $\sigma$ runs through the elements of $H_n$, and $\lambda_1, \ldots, \lambda_{k_n}$ run through the integers, the $\sigma+\sum_{i=1}^{k_n}\lambda_i\sigma_i$ are pairwise distinct. As a consequence, we have
$$e^{nk_n \eta}|H\cap B_n(1)_\R|\leq B_n(1+k_ne^{-n(\varepsilon-\eta)}).$$

Applying Proposition \ref{proposition:points-in-annulus} and noting that $k_n$ is bounded above by $r_n$, which is a polynomial in $n$, we get 
$$\frac{|B_n(1)\cap H_n|}{|B_n(1)|}=O(e^{-nk_n\eta}).$$
Together with (\ref{equation:bound-corank}), this shows the required estimate.

\bigskip

Now assume that $\cY$ is not flat over $\Z$. Since $\cY$ is irreducible, it lies over a closed point $p$ of $\Spec\Z$. By Proposition \ref{proposition:uniform-lower-bound-geometric}, we can find an integer $N$ and a constant $C$, independent of $\cY$ and $p$, such that for any $n\geq N$, the kernel of the restriction map
$$H^0(\cX_p, \mathcal L^{\otimes n})\ra H^0(\cY, \mathcal L^{\otimes n})$$
has codimension at least $Cn^d$ as a vector space over $\mathbb F_p$. Let $k_n$ be this codimension. Then 
\begin{equation}\label{equation:bound-corank-2}
k_n\geq Cn^{d}.
\end{equation}

Again, by Proposition \ref{proposition:exponentially-decreasing-sections}, up to enlarging $N$, we can find a positive number $\varepsilon$, depending only on $\cX$ and $\Lbar$, such that for any $n\geq N$, there exist sections $\sigma_1, \ldots, \sigma_{k_n}$ of $H^0(\cX, \mathcal L^{\otimes n})$ with $||\sigma_i||\leq e^{-n\varepsilon}$ for all $i\in\{1, \ldots, k_n\}$, such that the images of $\sigma_1, \ldots, \sigma_{k_n}$ in $H^0(\cY, \mathcal L^{\otimes n})$ are linearly independent over $\mathbb F_p$. 

Let $H_n$ be the kernel of the restriction map $\phi_{n, \cY}$. If $\sigma$ is an element of $H_n$, and if $\lambda_1, \ldots, \lambda_{k_n}$ are integers running through $\{0, \ldots, p-1\}$, then $\sigma+\lambda_1\sigma_1+\ldots+\lambda_{k_n}\sigma_{k_n}$ belongs to $H_n$ if and only if all the $\lambda_i$ vanish. Furthermore, the elements $\sigma+\lambda_1\sigma_1+\ldots+\lambda_{k_n}\sigma_{k_n}$ are pairwise disjoint. As a consequence, considering only those $\lambda_i$ that are $0$ or $1$, we have 
$$2^{k_n}|H_n\cap H^0_{Ar}(\cX, \Lbar^{\otimes n})|\leq |B_n(1+k_ne^{-n\varepsilon})|.$$
Again, applying Proposition \ref{proposition:points-in-annulus} and noting that $k_n$ is bounded above by $r_n$, which is a polynomial in $n$, we get 
$$\frac{|B_n(1)\cap H_n|}{|B_n(1)|}=O(e^{-k_n\eta}).$$
Together with (\ref{equation:bound-corank-2}), this shows the required estimate.

\end{proof}

\section{Variants and consequences}

\subsection{The irreducibility theorem over finite fields}

The arithmetic Bertini theorems we prove are stronger than their finite fields counterparts. Since the latter are known, we give only an example to illustrate how one can deduce them.

\begin{proposition}
Assume Theorem \ref{thm:main}. Let $k$ be a finite field, and let $X$ be an irreducible projective variety over $k$. Let $L$ be a very ample line bundle on $X$. Then the set 
$$\{\sigma\in \bigcup_{n>0} H^0(X, L^{\otimes n}), \,\mathrm{div}(\sigma)\,\textup{is irreducible}\}$$
has density $1$ in $\bigcup_{n>0} H^0(X, L^{\otimes n})$.
\end{proposition}

\begin{proof}
Since $L$ is very ample, we can find a positive integer $N$ a closed embedding $i : X\ra\mathbb P^N_k$ such that $L=i^*\mathcal O(1)$. 
Apply Theorem \ref{thm:main} to the composition 
$$ f : X\ra \mathbb P^N_k\ra \mathbb P^N_{\OK},$$
where $K$ is a number field together with a finite prime $\mathfrak p$ such that $\OK/\mathfrak p=k$, and the line bundle $\Lbar=\overline{\mathcal O_{\mathbb P^N_{\OK}}(1)}$ endowed with the Fubini-Study metric. The hermitian line bundle $\Lbar$ is the pullback of $\overline{\mathcal O(1)}$ by the map $\mathbb P^N_{\OK}\ra\mathbb P^N_\Z$, which is finite and unramified over the archimedean place. By Proposition \ref{proposition:FS-ample} and Corollary \ref{corollary:pull-back-ample}, $\Lbar$ is ample.

Since $f(X)$ is supported over a closed point of $\Spec\Z$, Theorem \ref{thm:main} guarantees that the set 
$$\{\sigma\in \bigcup_{n>0} H^0_{Ar}(\mathbb P^N_{\OK},\Lbar^{\otimes n}), \,\mathrm{div}(\sigma_{|X})\,\textup{is irreducible}\}$$
has density $1$ in $\bigcup_{n>0} H^0_{Ar}(\mathbb P^N_{\OK}, \Lbar^{\otimes n})$. By Corollary \ref{corollary:density-and-restriction}, the theorem holds.
\end{proof}

Note that since on a scheme $X$ defined over a finite field, every line bundle is a hermitian line bundle, and every section is effective, we can remove the flatness assumptions on the theorems of the introduction and have uniform statements that cover both the results of this paper and those of \cite{CharlesPoonen16}.

\subsection{Generic smoothness}

We first state the Bertini smoothness theorem of Poonen \cite{Poonen04} in the form we need -- see \cite{ErmanWood15} for the proof of this version. 

\begin{theorem}\label{theorem:bertini-smoothness}
Let $X$ be a smooth projective variety over a finite field $k$, and let $L$ be an ample line bundle on $X$. Then the density of those $\sigma\in\bigcup_{n>0} H^0(X, L^{\otimes n})$ such that $\mathrm{div}(\sigma)$ is smooth is equal to $\zeta(1+\mathrm{dim}(X))^{-1}$.
\end{theorem}

Applying the above result together with the restriction results of Corollary \ref{corollary:density-and-restriction}, we find the following. 

\begin{proposition}\label{proposition:smoothness-at-p}
Let $\cX$ be a projective arithmetic variety, and let $\Lbar$ be an ample hermitian line bundle on $\cX$. Let $p$ be a prime number such that $\cX_p$ is smooth over $\mathbb F_p$. Then the density of those $\sigma\in\bigcup_{n>0} H^0(X, L^{\otimes n})$ such that $\mathrm{div}(\sigma_{|\cX_p})$ is smooth is equal to $\zeta_p(\mathrm{dim}(\cX))^{-1}$, where $\zeta_p$ is the zeta function of $\cX_p$.
\end{proposition}

\begin{proof}[Proof of Theorem \ref{theorem:generic-smoothness}]
In the situation of the theorem, we know that $\cX_p$ is smooth for all large enough $p$. Furthermore, denoting again the zeta function of $\cX_p$ by $\zeta_p$, we have 
$$\lim_{p\ra\infty}\zeta_p(x)=1$$
for any $x>1$ by \cite[1.3]{Serre65}. This shows that the density of those $\sigma\in\bigcup_{n>0} H^0_{Ar}(\cX, \Lbar^{\otimes n})$ such that there exists $p$ with $\cX_p$ smooth and $\mathrm{div}(\sigma_{|\cX_p})$ smooth is equal to $1$. For any such $\sigma$, the divisor $\mathrm{div}(\sigma)_\Q$ is smooth, which proves the result.
\end{proof}

\subsection{Irreducibility theorems with local conditions}

We can give variants of the irreducibility theorems with conditions at prescribed subschemes. For an easier formulation, we give them in the setting of Theorem \ref{thm:X=Y}.

\begin{proposition}\label{proposition:condition-on-subschemes}
Let $\cX$ be a projective arithmetic variety, and let $\Lbar$ be an ample hermitian line bundle on $\cX$. Let $Z_1$ be a finite subscheme of $\cX$, and let $Z_2$ be a positive-dimensional subscheme of $\cX$. Choose a trivialization $\phi : \mathcal L_{|Z_1}\simeq \mathcal O_{Z_1}$, and let $T$ be a subset of $H^0(Z_1, \mathcal O_{Z_1})$. Then the density of those $\sigma\in\bigcup_{n>0} H^0_{Ar}(\cX, \Lbar^{\otimes n})$ such that $\sigma_{|Z_1}$ belongs to $T$ (under the trivialization $\phi$) and $\sigma$ does not vanish identically on any component of $Z_2$ is equal to 
$$\frac{|T|}{|H^0(Z_1, \mathcal O_{Z_1})|}.$$
\end{proposition}

\begin{proof}
By Corollary \ref{corollary:density-and-restriction}, the density of those $\sigma$ such that $\sigma_{|Z_1}$ belongs to $T$ is indeed $\frac{|T|}{|H^0(Z_1, \mathcal O_{Z_1})|}.$ On the other hand, Theorem \ref{theorem:independent-restriction} ensures that the density of those $\sigma$ that do not vanish identically on any component of $Z_2$ is equal to $1$.
\end{proof}

Given Theorem \ref{thm:X=Y} -- proven in the last section of this paper -- and Theorem \ref{theorem:generic-smoothness}, we find the two following results.

\begin{corollary}
Let $\cX$ be a projective arithmetic variety of dimension at least $2$, and let $\Lbar$ be an ample hermitian line bundle on $\cX$. Let $Z_1$ be a finite subscheme of $\cX$, and let $Z_2$ be a positive-dimensional subscheme of $\cX$. Choose a trivialization $\phi : \mathcal L_{|Z_1}\simeq \mathcal O_{Z_1}$, and let $T$ be a subset of $H^0(Z_1, \mathcal O_{Z_1})$. Then the density of those $\sigma\in\bigcup_{n>0} H^0_{Ar}(\cX, \Lbar^{\otimes n})$ such that the following condition hold:
\begin{enumerate}[(i)]
\item  $\sigma_{|Z_1}$ belongs to $T$ (under the trivialization $\phi$);
\item $\sigma$ does not vanish identically on any component of $Z_2$;
\item $\mathrm{div}(\sigma)$ is irreducible,
\end{enumerate}
is equal to 
$$\frac{|T|}{|H^0(Z_1, \mathcal O_{Z_1})|}.$$
\end{corollary}

\begin{corollary}\label{corollary:Taylor}
Let $\cX$ be a projective arithmetic variety with smooth generic fiber, and let $\Lbar$ be an ample hermitian line bundle on $\cX$. Let $Z_1$ be a finite subscheme of $\cX$, and let $Z_2$ be a positive-dimensional subscheme of $\cX$. Choose a trivialization $\phi : \mathcal L_{|Z_1}\simeq \mathcal O_{Z_1}$, and let $T$ be a subset of $H^0(Z_1, \mathcal O_{Z_1})$. Then the density of those $\sigma\in\bigcup_{n>0} H^0_{Ar}(\cX, \Lbar^{\otimes n})$ such that the following condition hold:
\begin{enumerate}[(i)]
\item  $\sigma_{|Z_1}$ belongs to $T$ (under the trivialization $\phi$);
\item $\sigma$ does not vanish identically on any component of $Z_2$;
\item $\mathrm{div}(\sigma)_\Q$ is smooth,
\end{enumerate}
is equal to 
$$\frac{|T|}{|H^0(Z_1, \mathcal O_{Z_1})|}.$$
\end{corollary}


\section{Preliminary estimates}

This section gathers preliminary material on hermitian line bundles on arithmetic surfaces, which will be used in the proof of Theorem \ref{thm:X=Y}. In \ref{subsection:norm-estimates}, we give lower bounds for the norm of products of sections of hermitian line bundles. In \ref{subsection:upper-bound-theta}, we give an upper bound for the number of effective sections of a hermitian line bundle in terms of its degree with respect to a positive enough hermitian line bundle. Such a result is closely related to the effective bounds of \cite{YuanZhang13}. Our proof is better expressed in terms of the $\theta$-invariants of Bost \cite{Bost15}, which we only consider in a finite-dimensional setting. In \ref{subsection:ample-cone}, we bound the number of effective hermitian line bundles satisfying certain boundedness properties. 

\subsection{Norm estimates for sections of hermitian line bundles}\label{subsection:norm-estimates}

Let $X$ be a compact connected Riemann surface. Let $\omega$ be a real $2$-form of type $(1,1)$ on $X$, which is positive except on a finite number of points of $X$ and satisfies 
$$\int_X\omega=1.$$
Define 
$$d^c=\frac{1}{2i\pi}(\partial-\overline\partial),$$
so that 
$$dd^c=\frac{i}{\pi}\partial\overline\partial.$$

Let $s$ be a section of a hermitian line bundle on $X$. In what follows, in order to avoid confusion, we will write $||s||$ for the function $P\mapsto ||s(P)||$ and $||s||_\infty$ for the sup norm of $s$.

Let $\overline L=(L, ||.||)$ is a hermitian line bundle on $X$. If $s$ is a nonzero section of $L$, the Lelong-Poincaré formula gives us the equality of currents
$$-dd^c\log||s||=c_1(\overline L)-\delta_D,$$
where $c_1(\overline L)$ is the curvature form of $\overline L$, $D$ is the divisor of $s$ and $\delta_D$ is the current of integration along $D$. 

Define, following \cite[(1.4.8)]{BostGilletSoule94}
$$||s||_0=\exp\Big(\int_X\log ||s||\omega\Big).$$
Since $\int_X \omega=1,$ the following inequality holds:
$$||s||_0\leq ||s||_\infty.$$

Say that $\overline L$ is \emph{admissible} if $c_1(\overline L)$ is proportional to $\omega$. If $\overline L$ is admissible, then the Gauss-Bonnet formula shows
$$c_1(\overline L)=(\deg L)\omega.$$

Let $M$ be any line bundle on $X$. By the $\partial\overline\partial$ lemma, we can find a hermitian metric $||.||$ on $M$ such that the hermitian line bundle $(M, ||.||)$ is admissible. Given a nonzero global section $s$ of $M$, there exists a unique such metric such that $||s||_0=1$.

If $D$ is an effective divisor on $X$, let $\sigma_D$ be the section of $\mathcal O(D)$ that is the image of $1$ under the natural morphism $\mathcal O_X\ra\mathcal O(D)$. The discussion above shows that there exists a unique admissible hermitian line bundle $\overline{\mathcal O(D)}=(\mathcal O(D), ||.||)$ on $X$ such that $||\sigma_D||_0=1$. Of course, if $D_1$ and $D_2$ are effective line bundles, we have 
$$\overline{\mathcal O(D_1+D_2)}=\overline{\mathcal O(D_1)}\otimes\overline{\mathcal O(D_2)}.$$

The functions $\sigma_P$ satisfy basic uniformities in $P$ which are readily proved by the following argument using Green functions.

\begin{proposition}\label{proposition:uniformity-green}
Endow $X$ with a Riemannian metric with induced geodesic distance $d$. Then there exist positive constants $C, C'$ and $\eta$ such that the following inequalities hold:
\begin{enumerate}[(i)]
\item $\forall P\in X,\, ||\sigma_P||_{\infty}\leq C;$
\item $\forall (P, Q)\in X\times X,\, \sigma_P(Q)\geq \min(C'd(P, Q), \eta).$
\end{enumerate}
\end{proposition}

\begin{proof}
Let $\Delta\subset X\times X$ be the diagonal. Let $\alpha$ be a real closed form of type $(1,1)$ on $X\times X$ of the form 
$$\alpha=p_1^*\omega+p_2^*\omega+\sum_{i\in I}p_1^*\beta_i\wedge p_2^*\gamma_i,$$
where $p_1$ and $p_2$ are the two projections from $X\times X$ to $X$ and the $\beta_i$ (resp. $\gamma_i$) are $1$-forms on $X$. Choose the $\beta_i$ and $\gamma_i$ so that $\alpha$ is symmetric with respect to the involution of $X\times X$ that exchanges the two factors, and that it is cohomologous to the class of the diagonal $\Delta$ in the de Rham cohomology of $X\times X$. By the $\partial\overline\partial$ lemma, we can find a hermitian metric on the line bundle $\mathcal O(\Delta)$ with curvature form $\alpha$. For any $P$ in $X$, this hermitian metric induces a hermitian metric on $\mathcal O(P)$ by restriction to $\{P\}\times X$. 

Let $\sigma_\Delta$ be the global section of $\mathcal O(\Delta)$ corresponding to the constant function $1$. For any $P$ in $X$, write $\tau_P$ for the section of $\mathcal O(P)$. 
$$\tau_P : Q\mapsto \sigma_\Delta(P, Q).$$
Then 
$$-dd^c\log ||\tau_P||=\alpha_{|\{P\}\times X}-\delta_P=\omega-\delta_P,$$
which shows that the metric on $\mathcal O(P)$ coming from that on $\mathcal O(\Delta)$ differs from the canonical one defined above by a homothethy. In particular, we can find a continuous function $X\ra \R^*_+, \,P\mapsto \lambda(P)$ such that 
$$\forall (P, Q)\in X\times X, ||\sigma_P(Q)|| =\lambda(P) ||\sigma_\Delta(P, Q)||.$$
Since $(P, Q)\mapsto \sigma_\Delta(P, Q)$ is a smooth section of $\mathcal O(\Delta)$ that vanishes with the order $1$ along $\Delta$, this shows the result\footnote{Actually, a straightforward computation shows that $||\tau_P||_0$ is a constant function of $P$, so that $\lambda$ is constant.}.\end{proof}

We will make use of the uniformity above to prove inequalities between norms. The following is a variant of \cite[Corollary 1.4.3]{BostGilletSoule94}.

\begin{proposition}\label{proposition:bound-low-degree}
Let $\overline L=(L, ||.||)$ be an admissible hermitian line bundle on $X$. Let $P$ be a point of $X$, and let $s$ be a section of $L$. Then 
$$||s(P)||\leq ||s||_0 \,||\sigma_P||_\infty^{\deg L}.$$
In particular, there exists a positive constant $C_1$ such that 
$$||s||_\infty\leq C_1^{\deg L}||s||_0.$$
\end{proposition}

\begin{proof}
We can assume that $s$ is nonzero.
Let $D$ be the divisor of $s$. Define 
$$g=-\log ||s||$$
and 
$$g_P=-\log ||\sigma_P||.$$
By Lelong-Poincaré, we have 
$$dd^c g=(\deg L)\omega-\delta_D$$
and 
$$dd^c g_P=\omega-\delta_P.$$
The Stokes formula 
$$\int_X g\,dd^c g_P=\int_X g_P\, dd^c g$$ 
gives us
$$-\log ||s||_0+\log ||s(P)||=-\deg L\log ||\sigma_P||_0+\log||\sigma_P(D)||=\log||\sigma_P(D)||,$$
where, if $D=\sum_i n_iP_i$, we wrote 
$$||\sigma_P(D)||=\Pi_i||\sigma(P_i)^{n_i}||.$$
Since the degree of $D$ is equal to the degree of $L$, we get the first inequality. The second one follows from the first and Proposition \ref{proposition:uniformity-green}.
\end{proof}

\begin{lemma}\label{lemma:comparison-2}
Let $\overline L=(L, ||.||)$ be an admissible hermitian line bundle on $X$ with positive degree. Then for any section $s$ of $L$, and any $P$ in $X$, the following inequality holds:
$$||s||_\infty\leq C_2(\deg L) ||s\sigma_P||_\infty,$$
where $C_2$ is a positive constant depending only on $X$ and $\omega$.
\end{lemma}

\begin{proof}
Let $B$ be the ball $\{z\in\C|\,|z|<3\}$. Let $(U_i)_{i\in I}$ be a finite cover of $X$ by open subsets such that there exist biholomorphisms
$$f_i : B \ra U_i$$
and assume that $X$ is covered by the $f_i(\{z\in\C|\,|z|<1\})$. For all $i\in I$, choose a smooth function $\phi_i : B\ra\R$ such that $-dd^c\phi_i=f_i^*\omega$.

Since $\overline L$ is admissible, the curvature form of $\overline L$ is $(\deg L)\omega$ and for any $i\in I$, we can find an isomorphism of hermitian line bundles on $B$
$$f_i^*\overline L\simeq (\mathcal O, e^{-\lambda\phi_i}),$$
where $|.|$ is the standard absolute value and $\lambda=\deg L$.

\bigskip

Choose an element $i\in I$, and a complex number $z$ with $|z|<1$ such that 
$$|f_i^*s(z)|e^{-\lambda\phi_i(z)}=||s||_\infty,$$
where $f_i^*s(z)$ is considered as a complex number via the isomorphism above. By Proposition \ref{proposition:uniformity-green}, we can find positive constants $\varepsilon<1$ and $\eta$ depending only on $X$ and $\omega$ such that either $|\sigma_P(z)|\geq \frac{\eta}{\lambda}$ or, 
$$\forall z'\in\C, |z'-z|=\frac{\varepsilon}{\lambda}\implies |\sigma_P(z')|\geq \frac{\eta}{\lambda}.$$
If $\sigma_P(z)\geq\frac{\eta}{\lambda}$, then 
$$||s\sigma_P||_\infty\geq ||s\sigma_P(f_i(z))||\geq\frac{\eta}{\lambda}||s(f_i(z))||=\frac{\varepsilon}{\deg L}||s||_\infty.$$
Else, we can find a complex number $z'$ with $|z'-z|=\frac{\varepsilon}{\lambda}$ and $|f_i^*s(z')|\geq |f_i^*s(z)|.$ Then 
$$||s\sigma_P||_\infty\geq ||s\sigma_(f_i(z'))||\geq \frac{\eta}{\lambda}||s||_\infty e^{-\lambda(\phi_i(z)-\phi_i(z'))}\geq e^{-C}\frac{\eta}{\deg L}||s||_\infty,$$
where $C$ is an upper bound for the differential of the $\phi_i$ on the ball $\{z\in\C|\,|z|<2\}$ as $i$ varies through the finite set $I$. 
\end{proof}

\begin{proposition}\label{proposition:bound-large-degree}
Let $\overline L=(L, ||.||)$ and $\overline M=(M, ||.||)$ be two admissible hermitian line bundles on $X$. Then for any two sections $s$ and $\sigma$ of $L$ and $M$ respectively, the following inequality holds:
$$||\sigma||_0||s||_\infty\leq (C_2(\deg L + \deg M))^{\deg M} ||s\sigma||_\infty,$$
where $C_2$ is a positive constant depending only on $X$ and $\omega$.
\end{proposition}

\begin{proof}
Let $D$ be the divisor of $\sigma$. Then $M$ is isomorphic to $\mathcal O(D)$, and the hermitian line bundles $\overline M$ and $\overline{\mathcal O(D)}$, as well as the sections $\sigma$ and $\sigma_D$ differ by a homothethy. Since the inequality we want to prove is invariant under scaling, we can assume that $\overline M=\overline{\mathcal O(D)}$ and $\sigma=\sigma_D$. If $D=\sum_i n_iP_i$, we have 
$$\overline{\mathcal O(D)}=\bigotimes_i\overline{\mathcal O(P_i)}^{\otimes n_i}$$
and 
$$\sigma_D=\Pi_i \sigma_{P_i}^{n_i},$$
so that the result follows from successive applications of Lemma \ref{lemma:comparison-2}.
\end{proof}

\subsection{An upper bound for the number of sections}\label{subsection:upper-bound-theta}
Let $\cX$ be a projective arithmetic variety with smooth generic fiber. Choose a K\"ahler form on $\cX(\C)$ which is invariant under complex conjugation and has volume $1$. If $\Lbar$ is a hermitian line bundle on $\cX$, we write $h^0_\theta(\cX, \Lbar)$ for $h^0_\theta(H^0_{L^2}(\cX, \Lbar))$, where the hermitian vector bundle $H^0_{L^2}(\cX, \Lbar)$ over $\Spec \Z$ is endowed with the $L^2$ norm induced by the K\"ahler metric on $\cX$.

We will need a comparison result between the sup norm and the $L^2$ norm on the space of sections of hermitian line bundles, which we will obtain through a minor generalization of Gromov's lemma \cite[Lemma 30]{GilletSoule92}. We follow the proof of Gillet-Soulé and start with a local result. 

\begin{lemma}\label{lemma:local-estimate-gromov}
Let $d$ be a positive integer, and let $B$ be the open ball $\{z\in \C^d |\, |z|<3\}$ in $\C^d$. Let $\phi$ be a real-valued smooth function on $B$, and let $g$ be a smooth positive function on $B$. Then there exists a positive constant $C$ depending only on $\phi$ and $g$ such that for any $\lambda\geq 1$, any holomorphic function $f$ on $B$ and any $w$ in $B$ with $|w|<1$, 
$$\int_{|z-w|<1} |f(z)|^2 e^{-2\lambda\phi(z)}g(z) dx_1\ldots dy_d\geq C |f(w)|^2e^{-2\lambda\phi(w)}\lambda^{-2d}.$$
\end{lemma}

\begin{proof}
If $\lambda$ is an integer, the inequality is the "local statement" proved in the beginning of the proof of \cite[Lemma 30]{GilletSoule92}. 

To prove our result, after adding a constant to $\phi$, we can assume that $\phi$ is negative on the ball $|z|<2$. Let $C'$ be an lower bound for the values of $\phi$ on the ball $|z|<2$. If $\lambda>1$ is arbitrary, write $\lambda=n+r$, with $0\leq r<1$. Then 
$$e^{-2\lambda\phi(z)}=e^{-2n\phi(z)}e^{-2r\phi(z)}\geq e^{-2n\phi(z)}$$
for any $z$ with $|z|<2$, so that 
\begin{align*}
\int_{|z-w|<1} |f(z)|^2 e^{-2\lambda\phi(z)} g(z)dx_1\ldots dy_d & \geq \int_{|z-w|<1} |f(z)|^2 e^{-2n\phi(z)} g(z)dx_1\ldots dy_d\\
&\geq C |f(w)|^2e^{-2n\phi(w)}n^{-2d}\\
&\geq C |f(w)|^2e^{-2\lambda\phi(w)}e^{2r\phi(w)}\lambda^{-2d}\\
\ &\geq Ce^{2C'}|f(w)|^2e^{-2\lambda\phi(w)}\lambda^{-2d}.
\end{align*}
Replacing $C$ with $Ce^{2C'}$, we get the result.
\end{proof}

\begin{proposition}\label{proposition:gromov-for-admissible}
Let $X$ be a compact connected riemannian complex manifold of dimension $d$, let $\omega$ be a real form of type $(1,1)$ on $X$. Then there exists a positive constant $C$ such that for any hermitian line bundle $\Lbar$ on $X$ with positive degree and curvature form $\lambda\omega$ with $|\lambda|>1$, and any section $s$ of $\mathcal L$ over $X$, we have
$$||s||_{L^2}\geq C|\lambda|^{-2d}||s||,$$
where $||s||_{L^2}$ denotes the $L^2$ norm of $s$ with respect to the given metric on $X$.

In particular, if $d=1$, there exists a positive constant $C'$ such that  for any hermitian line bundle $\Lbar$ with curvature form proportional to $\omega$ and positive degree, and any section $s$ of $\mathcal L$, we have 
$$||s||_{L^2}\geq C'(\deg\mathcal L)^{-2}||s||.$$
\end{proposition}

\begin{proof}
As above, let $B$ be the open ball $\{z\in \C^d |\, |z|<3\}$ in $\C^d$. Let $(U_i)_{i\in I}$ be a finite cover of $X$ by open subsets such that there exists biholomorphisms 
$$f_i : B\ra U_i$$
and assume that $X$ is covered by the $f_i(\{z\in \C^d|\,|z|<1\})$. For any $i\in I$, we can find a positive smooth function $g_i$ such that the pullback of the standard metric of $X$ to $B$ by $f_i$ is $g_idx_1\ldots dy_d$.

For all $i\in I$, choose a function $\phi_i$ on $B$ such that $-dd^c\phi_i=f_i^*\omega$. Let $\Lbar$ be a hermitian line bundle with curvature form $\lambda\omega$ for some real number $\lambda$ with $|\lambda|>1$. Then, for any $i\in I$, we can fix an isomorphism of hermitian line bundles
$$f_i^*\Lbar\simeq (\mathcal O_B, e^{-\lambda\phi_i}|.|)$$
where $|.|$ is the standard absolute value. Applying Lemma \ref{lemma:local-estimate-gromov} (up to replacing $\phi_i$ by $-\phi_i$ if $\lambda$ is negative), we can find a positive constant $C$, independent of $\Lbar$, such that, given any section $s$ of $\Lbar$, for any $i\in I$ and any $w$ in $B$ with $|w|<1$, we have 
$$\int_{|z-w|<1} |f_i^*s(z)|^2 e^{-\lambda\phi_i(z)}g(z) dx_1\ldots dy_d\geq C |f_i^*s(w)|^2e^{-\lambda\phi_i(w)}|\lambda|^{-2d},$$
where we consider $f_i^*s$ as a holomorphic function via the local trivializations of $\mathcal L$. This inequality means
$$\int_{|z-w|<1} ||s(f_i(z))||^2 g(z) dx_1\ldots dy_d\geq C |\lambda|^{-2d}||s(f_i(w))||^2,$$
so that 
$$||s||_{L^2}\geq \int_{|z-w|<1} ||s(f_i(z))||^2 g(z) dx_1\ldots dy_d\geq C |\lambda|^{-2d}||s(f_i(w))||^2$$
for any $w$, which proves the first result. 

The second result is a consequence of the first one and the Gauss-Bonnet formula. 
\end{proof}

Given a real form $\omega$ of type $(1, 1)$ on a complex manifold, say that a hermitian line bundle $\overline L$ is $\omega$-admissible if its curvature form is a multiple of $\omega$. We write $\widehat{\mathrm{Pic}}_\omega(\cX)$ for the group of $\omega$-admissible hermitian line bundles on $\cX$.

\begin{proposition}\label{proposition:upper-bound-theta}
Let $\cX$ be a regular projective arithmetic surface. Choose a K\"ahler form on $\cX(\C)$ which is invariant under complex conjugation, and let $\Bbar$ be a hermitian line bundle on $\cX$. Let $\omega$ be a real form of type $(1, 1)$ on $\cX(\C)$ with $\int_{\cX(\C)}\omega\neq 0$. Assume that the following conditions hold:
\begin{enumerate}[(i)]
\item Some positive power of $\Bbar$ is effective;
\item $\Bbar.\Bbar>0$;
\item If $\Mbar$ is an effective hermitian line bundle on $\cX$, then $\Bbar.\Mbar\geq 0.$
\end{enumerate}
Then for any effective $\Mbar\in\widehat{\mathrm{Pic}}_\omega(\cX)$, we have
$$h^0_\theta(\cX, \Mbar)\leq \frac{(\Bbar.\Mbar)^2}{2\Bbar.\Bbar}+O(\Mbar.\Bbar\log(\Mbar.\Bbar))+O(\deg\mathcal M_\Q\log(\deg\mathcal M_\Q))+O(1),$$
where the implied constants depend on $\cX$, $\Bbar$ and $\omega$, but not on $\Mbar$.
\end{proposition}

\begin{remark}
If $\Mbar$ is effective, then both $\Bbar.\Mbar$ and $\deg\mathcal M_\Q$ are nonnegative.
\end{remark}

\begin{proof}
%

Let $\Mbar$ be an effective, $\omega$-admissible, hermitian line bundle. If $\mathcal M_{\Q}$ has degree zero, then the curvature form of $\Mbar$ vanishes, so that $\Mbar$ is isomorphic to $\overline{\mathcal O}_{\cX}$ and the inequality of the proposition holds. We can assume that the degree of $\mathcal M_\Q$ is positive. Let us write $d$ for the degree of $\mathcal M_\Q$. 

After replacing $\Bbar$ by a positive power, we can assume that $\Bbar$ is effective. Let $s$ be a nonzero effective section of $\Bbar$ with divisor $D$. We have an exact sequence of line bundles
$$0\ra\mathcal M\otimes\mathcal B^{\otimes -1}\ra \mathcal M\ra \mathcal M_{|D}\ra 0.$$
Taking global sections, we get an exact sequence
$$0\ra H^0(\cX, \mathcal M\otimes \mathcal B^{\otimes -1})\ra H^0(\cX, \mathcal M)\ra H^0(D, \mathcal M_{|D}).$$
The map of lattices
$$ i : H^0_{L^2}(\cX, \Mbar \otimes \Bbar^{\otimes -1})\ra H^0_{L^2}(\cX, \Mbar)$$
is the multiplication by a section of sup norm bounded above by $1$, so the norm of $i$ is bounded above by $1$.

If $s$ is a global section of $\Mbar$ on $\cX$, then certainly we have, for the sup norms.
$$||s||\geq ||s_{|D}||.$$
Endow $H^0(D, \Mbar_{|D})$ with the $L^2$ norm
$$||t||^2_{L^2}=\sum_{z\in D(\C)} ||t(z)||^2$$
for $t\in H^0(D, \Mbar_{|D})$. Then for any section of $\mathcal M$ over $D$, we have
$$||s||^2\geq \frac{1}{\deg D_{\Q}}||s_{|D}||^2_{L^2},$$
so that Proposition \ref{proposition:gromov-for-admissible} shows that the norm of the map of lattices
$$H^0_{L^2}(\cX, \Mbar)\ra H^0_{L^2}(D, \Mbar_{|D})$$
has norm bounded above by $Cd^{2}$, where $C$ is a positive constant independent of $\Mbar$. In other words, the map of lattices 
$$H^0_{L^2}(\cX, \Mbar)\ra H^0_{L^2}(D, \Mbar_{|D})(\log C+2\log d)$$
has norm at most $1$ -- here if $\Lambda$ is a lattice and $\delta$ a real number, we write $\Lambda(\delta)$ for the lattice $\Lambda$ with the metric scaled by $e^{-\delta}$. Note that from \cite[Corollary 3.3.5, (2)]{Bost15}, we have 
$$h^0_{\theta}(H^0_{L^2}(D, \Mbar_{|D})(\log C+2\log d))\leq h^0_\theta(D, \Mbar_{|D})+\deg D_{\Q}(\log C+2\log d).$$

From the monotonicity and the subadditivity of $\theta$-invariants proved in \cite[Proposition 3.3.2, Proposition 3.8.1]{Bost15}, we get
\begin{equation}\label{equation:bound-curve}
h^0_\theta(\cX, \Mbar)\leq h^0_\theta(\cX, \Mbar\otimes\Bbar^{\otimes -1})+h^0_\theta(D, \Mbar_{|D})+O(\log d)+O(1),
\end{equation}
where the implied constants are independent of $\Mbar$.

By \cite[Proposition 3.7.1, Proposition 3.7.2]{Bost15}, we have\footnote{In \cite{Bost15}, hermitian vector bundles are only considered over the ring of integers of number fields. However, this assumption is irrelevant, and can be removed by considering the pullback of $\Mbar_{|D}$ to the normalization of $D$.}
$$h^0_\theta(D, \Mbar_{|D})\leq \max(\deg \Mbar_{|D}, 0)+O(1)\leq \Mbar.\Bbar+O(1)$$
since we assumed that $\Mbar.\Bbar\geq 0$ and since $D$ is the zero locus of an effective section of $\Bbar$. Together with (\ref{equation:bound-curve}), we obtain
\begin{equation}\label{equation:bound-induction}
h^0_\theta(\cX, \Mbar)\leq h^0_\theta(\cX, \Mbar\otimes\Bbar^{\otimes -1})+\Mbar.\Bbar+O(\log d)+O(1).
\end{equation}
Now let $m$ be the smallest integer such that $m\Bbar.\Bbar>\Lbar.\Bbar$, so that 
$$m\leq \lfloor \Lbar.\Bbar/\Bbar.\Bbar\rfloor+1.$$
Applying the argument above inductively to $\Lbar.\Bbar^{\otimes -i}$ as $i$ runs from $0$ to $m-1$, we get
\begin{equation}
h^0_\theta(\cX, \Mbar)\leq h^0_\theta(\cX, \Mbar\otimes\Bbar^{\otimes -m})+\frac{(\Mbar.\Bbar)^2}{2\Bbar.\Bbar}+O(\Mbar.\Bbar\log d)+O(\Mbar.\Bbar)+O(1).
\end{equation}
By construction, $\Bbar.(\Mbar\otimes\Bbar^{\otimes -m})<0$, so that condition $(iii)$ ensures that $\Mbar\otimes\Bbar^{\otimes -m}$ is not effective. By \cite[Corollary 4.1.2]{Bost15}, we get
$$h^0_\theta(\cX, \Mbar\otimes\Bbar^{\otimes -m})\leq O(d\log d)+O(1)$$
since the rank of $H^0(\cX, \mathcal M\otimes\mathcal B^{\otimes -m})$ is certainly bounded above by $O(d).$

Finally, we have 
\begin{equation}
h^0_\theta(\cX, \Mbar)\leq \frac{(\Mbar.\Bbar)^2}{2\Bbar.\Bbar}+O(\Mbar.\Bbar\log d)+O(d\log d)+O(\Mbar.\Bbar)+O(1),
\end{equation}
which shows the result.
\end{proof}

\begin{corollary}\label{corollary:upper-bound-arakelov}
Let $\cX$ be a regular projective arithmetic surface, and let $\Bbar$ be a hermitian line bundle on $\cX$. Let $\omega$ be a real form of type $(1, 1)$ on $\cX(\C)$ with $\int_{\cX(\C)}\omega\neq 0$. Assume that the following conditions hold:
\begin{enumerate}[(i)]
\item Some positive power of $\Bbar$ is effective;
\item $\Bbar.\Bbar>0$;
\item If $\Mbar$ is an effective hermitian line bundle on $\cX$, then $\Bbar.\Mbar\geq 0.$
\end{enumerate}
Then for any effective $\Mbar\in\widehat{\mathrm{Pic}}_\omega(\cX)$, we have
$$h^0_{Ar}(\cX, \Mbar)\leq \frac{(\Bbar.\Mbar)^2}{2\Bbar.\Bbar}+O(\Mbar.\Bbar\log(\Mbar.\Bbar))+O(\deg\mathcal M_\Q\log(\deg\mathcal M_\Q))+O(1),$$
where the implied constants depend on $\cX$ and $\Bbar$, but not on $\Mbar$.
\end{corollary}

\begin{proof}
From Proposition \ref{proposition:upper-bound-theta} and \cite[Theorem 4.1.1]{Bost15}, we find that the inequality holds if one replaces $h^0_{Ar}(\cX, \Mbar)$ with $h^0_{Ar, L^2}(\cX, \Mbar)$ -- this being the logarithm of the number of sections of $\Mbar$ with $L^2$ norm bounded above by $1$. Choosing the K\"ahler form on $\cX$ to have volume $1$, we have 
$$h^0_{Ar}(\cX, \Mbar)\leq h^0_{Ar, L^2}(\cX, \Mbar),$$
which finishes the proof.
\end{proof}

\begin{remark}
In \cite[Theorem A]{YuanZhang13}, Yuan and Zhang prove an explicit upper bound for $h^0_{Ar}(\cX, \Mbar)$ from which one can deduce -- via log-concavity of volumes -- special cases of our inequality.
\end{remark}

\subsection{An upper bound for the number of effective hermitian line bundles}\label{subsection:ample-cone} 

\begin{lemma}\label{lemma:ample-is-nef}
Let $\cX$ be a projective arithmetic surface, and let $\Lbar$ be an ample hermitian line bundle on $\cX$. If $\Mbar$ is an effective line bundle on $\cX$ which is not isomorphic to $\overline{\mathcal O}_{\cX}$, then 
$$\Lbar.\Mbar> 0.$$
\end{lemma}

\begin{proof}
Let $s$ be an effective section of $\Mbar$, and let $D$ be the divisor of $s$. Then by the formula \cite[(6.3.2)]{Deligne85}, we have 
$$\Lbar.\Mbar=h_{\Lbar}(D)-\int_{\cX(\C)}\log||s_\C||c_1(\Lbar),$$
where $h_{\Lbar}$ denotes the height with respect to $D$. The first term is nonnegative since $\Lbar$ is ample, and vanishes if and only if $D=0$. Since $s$ is effective and the curvature form of $\Lbar$ is positive, the second term is nonnegative as well, and vanishes if and only if the norm of $s$ is identically $1$. As a consequence, for $\Lbar.\Mbar$ to vanish, it is necessary for $\Mbar$ to have a nowhere vanishing section of norm identically $1$, i.e., to be isomorphic to $\overline{\mathcal O}_{\cX}.$
\end{proof}

\begin{proposition}\label{proposition:cone}
Let $\cX$ be a projective arithmetic surface, and let $\Lbar$ be an ample hermitian line bundle on $\cX$. Let $\omega$ be a real form of type $(1, 1)$ on $\cX(\C)$ with $\int_{\cX(\C)}\omega\neq 0$. Let $N$ be a subgroup of the group $\widehat{\mathrm{Pic}}_\omega(\cX)$ of $\omega$-admissible hermitian line bundles on $\cX$. Assume that the intersection of $N$ with $\Ker(\widehat{\mathrm{Pic}}_\omega(\cX)\ra \mathrm{Pic}(\cX))\simeq \R$ is discrete. Then $N$ is a group of finite type and if $N_{\mathrm{eff}}$ denotes the subspace of $N$ consisting of effective line bundles, as $n$ tends to $\infty$, we have
$$|\{\Mbar\in N_{\mathrm{eff}}|\,\Lbar.\Mbar\leq n \}=O(n^\rho).$$ 
\end{proposition}

\begin{proof}
The abelian group $\mathrm{Pic}(\cX)$ is finitely generated by \cite{Roquette57} -- see \cite{Kahn06} for a modern proof -- so that the image of $N$ in $\mathrm{Pic}(\cX)$ is a group of finite type. Since the intersection of $N$ with $\Ker(\widehat{\mathrm{Pic}}_\omega(\cX)\ra \mathrm{Pic}(\cX))$ is discrete, it is of finite type as well, which proves that $N$ is a group of finite type.

\bigskip

The linear form on $N$
$$\Mbar\mapsto\Lbar.\Mbar$$
extends to a linear form on $N_\R:=N\otimes\R$ which we still denote by 
$$\alpha\mapsto\Lbar.\alpha.$$
Let $\overline{N}_{\mathrm{eff}}$ be the closure of $N_{\mathrm{eff}}$ in $N\otimes \R$. Lemma \ref{lemma:ample-is-nef} shows that the linear form above is nonnegative on $N_{\mathrm{eff}}$, so it is nonnegative on $\overline{N}_{\mathrm{eff}}$. 

Our assumption on $N$ guarantees that the first chern class map
$$c_1 : \widehat{\mathrm{Pic}}_\omega(\cX)\ra \widehat{\mathrm{CH}}^1(\cX)$$
extends to an injection 
$$c_{1, \R} : N\otimes\R\ra \widehat{\mathrm{CH}}^1_\R(\cX),$$
where $\widehat{\mathrm{CH}}^1_\R(\cX)$ is the arithmetic Chow group with real coefficients defined in \cite[5.5]{Bost99}. Indeed, we have an exact sequence 
$$0\ra (N\cap\Ker(\widehat{\mathrm{Pic}}_\omega(\cX)\ra \mathrm{Pic}(\cX)))\otimes\R \ra N\otimes\R\ra \mathrm{Pic}(\cX)\otimes\R,$$
and the first term can be identified with $\R$ by assumption.

By the Hodge index theorem of Faltings \cite{Faltings84} and Hriljac \cite{Hriljac85} as stated in \cite[Theorem 5.5, (2)]{Bost99}, the intersection pairing on $\widehat{\mathrm{CH}}^1_\R(\cX)$ is non-degenerate. In particular, if $x$ be a nonzero element of $\overline{N}_{\mathrm{eff}},$ we can find a hermitian line bundle $\Mbar$ with $\Mbar.x<0$. If $n$ is a large enough integer, Corollary \ref{corollary:ampleness-open} shows that $\Lbar^{\otimes n}\otimes \Mbar$ is ample, so that the discussion above guarantees the inequality 
$$(\Lbar^{\otimes n}\otimes \Mbar).x=n\Lbar.x+\Mbar.x\geq 0.$$
This shows that $\Lbar.x$ is positive. 

The linear form $x\mapsto\Lbar.x$ is positive on the complement of the origin in the closed cone $\overline{N}_{\mathrm{eff}}.$ As a consequence, the number of integral points $x$ of $\overline{N}_{\mathrm{eff}}$ with $\Lbar.x\leq n$ is bounded above by a quantity of the form $O(n^\rho)$, where $\rho$ is the rank of $N$.
\end{proof}

%
%
%

\section{Irreducible ample divisors on arithmetic surfaces}

\subsection{Setup}\label{subsection:setup} In this section, we prove Theorem \ref{thm:X=Y} for arithmetic surfaces. 

Let $f : \cX\ra\Spec\Z$ be a projective arithmetic surface, and $\Lbar$ an ample line bundle on $\cX$. If $n$ is a large enough integer, we want to give an upper bound for the number of sections of $\Lbar^{\otimes n}$ that define a divisor which is not irreducible. We will give three different bounds that depend on the geometry and the arithmetic of the irreducible components of that divisor. 

In the statement below, $\cX$ is not assumed to be regular, but heights are still well-defined, see \cite[(1.2)]{Zhang95}.

\begin{proposition}\label{proposition:divisors-with-low-height}
Let $\alpha$ be a real number with $0<\alpha<\frac{1}{2}$. If $n$ is an integer, the proportion of those elements $s$ of $H^0_{Ar}(\cX, \Lbar^{\otimes n})$ that vanish on some Weil divisor $D$ of $\cX$ with $h_\Lbar(D)\leq n^{\alpha}$ goes to zero as $n$ goes to infinity.
\end{proposition}

\begin{proof}
Assume that $n$ is large enough. By \cite[Theorem B]{Moriwaki04}, the number of divisors $D$ on $\cX$ with $h_\Lbar(D)\leq  n^{\alpha}$ is bounded above by $e^{Cn^{2\alpha}}$ for some positive constant $C$. By Theorem \ref{theorem:independent-restriction}, we can find positive constants $C$ and $\eta$ such that for any $D$ as above, the proportion of those elements $s$ of $H^0_{Ar}(\cX, \Lbar^{\otimes n})$ that vanish on $D$ is bounded above by $C'e^{-n\eta}$. 

As a consequence, the proportion of those $s$ that vanish on any $D$ with $h_\Lbar(\cX)\leq n^{\alpha}$ is bounded above by 
$$C'e^{Cn^{2\alpha}-n\eta},$$
which goes to zero as $n$ goes to infinity.
\end{proof}

\subsection{Degree bounds and reduction modulo $p$}\label{subsection:finite-fields}

Let $f : \cX\ra\Spec\Z$ be as above. We want to investigate irreducible divisors on the fibers of $f$ above closed points and derive global consequences. Our goal here is to prove Proposition \ref{proposition:geometric-irreducibility}. 

\bigskip

Since $\cX$ is reduced, we can find a non-empty open subset $S$ of $\Spec\Z$ such that the restriction $f_S : \cX_S\ra S$ has reduced fibers. 

Let $r$ be the number of irreducible components of the geometric generic fiber of $f$. Up to shrinking $s$, we may assume that if $\overline s$ is any geometric point of $S$, then the number of irreducible components of $\cX_s$ is exactly $r$. Since $\cX_s$ is reduced by assumption, this is equivalent to the fact that the specialization map induces a bijection between the components of $\cX_{\overline \Q}$ and those of $\cX_{\overline s}$.

The degree of $\mathcal L_{\Q}$ is of the form $rd$, where $d$ is the degree of the restriction of $\mathcal L$ to a component of $\cX_{\overline \Q}$. Write $\mathcal L_p$ for the restriction of $\mathcal L$ to $\cX_p$. 

If $X$ is a reduced scheme, and $C$ is an irreducible component of $X$, we will always consider $C$ to be a subscheme of $X$, endowed with the reduced structure.

\begin{lemma}\label{lemma:Riemann-Roch}
Let $C$ be an integral projective curve over a perfect field, with arithmetic genus $p_a(C)$. Let $\mathcal L$ be a line bundle on $C$. Then 
$$h^0(C, \mathcal L)\geq 1-p_a(C)+\mathrm{deg}(\mathcal L)$$
and equality holds if the degree of $\mathcal L$ is strictly bigger than $p_a(C)$.
\end{lemma}

\begin{proof}
The first statement follows directly from the Riemann-Roch theorem. To prove the second one, consider the normalization $\pi : \widetilde C\ra C$ of $C$. Then $\widetilde C$ is smooth over the base field $k$, and its genus is bounded above by the arithmetic genus $p_a(C)$ of $C$. 

Since $C$ is reduced, it is Cohen-Macaulay, so that the dualizing sheaf $\omega_{C/k}$ of $C$ is Cohen-Macaulay by \cite[\href{http://stacks.math.columbia.edu/tag/0BS2}{Tag 0BS2}]{stacks-project}. In particular, it is torsion-free, so that the morphism $\pi^*\omega_{C/k}\ra \omega_{\widetilde C/k}$ is injective. Now assume that the degree of $\mathcal L$ is strictly bigger than $p_a(C)$. In particular, we have
$$\mathrm{deg}(\pi^*\mathcal L)>\mathrm{deg}(\omega_{\widetilde C/k})$$
and 
$$h^1(C, \mathcal L)=h^0(C, \mathcal L^\vee\otimes_{\mathcal O_C}\omega_{C/k})\leq h^0(\widetilde C, \pi^*\mathcal L^\vee\otimes_{\mathcal O_{\widetilde C}}\omega_{\widetilde C/k})=0.$$
By Riemann-Roch, we have 
$$h^0(C, \mathcal L)=\chi(\mathcal L)=1-p_a(C)+\mathrm{deg}(\mathcal L).$$
\end{proof}

\begin{lemma}\label{lemma:Lang-Weil}
Let $p$ be a prime number corresponding to a point in $S$, and let $\overline\Fp$ be an algebraic closure of $\Fp$. If $C$ is an irreducible component of $\cX_p$, let $r_C$ be the number of irreducible components of $C_{\overline\Fp}$, and if $k$ is a positive integer, let $N_k(C)$ be the number of irreducible divisors of degree $k$ on $C$. Then the following inequality holds:
$$\big|N_{r_Ck}(C)-\frac{1}{k}p^{r_Ck}\big|=O(p^{\frac{r_Ck}{2}}),$$
where the implied constants only depends on $f_S: \cX_S\ra S$.
\end{lemma}

\begin{proof}
The $r_C$ irreducible components of $C_{\overline\Fp}$ are all defined over $\mathbb{F}_{p^r}$. Denote them by $C_1, \ldots, C_{r_C}$. The Lang-Weil estimates of \cite{LangWeil54} give us the inequality, for any positive integer $k$:
$$\big| |C_1(\mathbb{F}_{p^{r_Ck}})|-p^{r_Ck}\big|=O(p^{\frac{r_Ck}{2}}),$$
where the implied constants only depends on the degree of an embedding of $C_1$ into some projective space -- in particular, it only depends on $f_S$. As a consequence, if $M_k$ is the number of elements in $C_1(\mathbb{F}_{p^{r_Ck}})$ with residue field exactly $\mathbb{F}_{p^{r_Ck}}$, we have:
$$\big | M_k-p^{r_Ck}\big | \leq \sum_{i|k, i\neq k} p^{r_Ci} +O(\sum_{i|k} p^{\frac{r_Ci}{2}})=O(kp^{\frac{r_Ck}{2}}).$$

Now assume that $r_Ck$ is strictly larger than the degree of the residue field of any singular point of $C$ -- this degree can be bounded independently of $C$ as $f_S$ is generically smooth. Irreducible divisors of degree $r_Ck$ on $C$ are in one-to-one correspondence irreducible divisors of degree $k$ on $C_1$, which in turn are in one-to-one correspondence with Galois orbits over $\mathbb F_{p^{r_C}}$ of elements of $C_1(\mathbb{F}_{p^{kr_C}})$. As a consequence, we have 
$$N_k(C)=\frac{1}{k}M_k,$$
which proves the estimate of the lemma.
\end{proof}

\begin{lemma}\label{lemma:vanishing-independent}
Let $p$ be a prime number corresponding to a point in $S$, and let $C$ be an irreducible component of $\cX_p$. Then there exists a positive integer $N$, independent of $C$ and $p$, such that for any $n\geq N$, the restriction map 
$$H^0(\cX_p, \mathcal L_p^{\otimes n})\ra H^0(C, \mathcal L_p^{\otimes n})$$
is surjective.
\end{lemma}

\begin{proof}
Since the result certainly holds if $N$ is allowed to depend on $p$ by general vanishing results for ample line bundle, we may replace $S$ by any nonempty open subset, which we will do along the proof.

Choose a finite flat map $S'\ra S$ such that the irreducible components of the generic fiber of $\cX_{S'}\ra S$ are geometrically ireducible. In particular, our assumption on $s$ guarantees that the irreducible components of the fiber of $\cX_{S'}\ra S'$ over any closed point $s'$ are geometrically irreducible, and are the intersection of an irreducible component of $\cX_{S'}$ with $\cX_{s'}$.

Let $s'$ be a point of $S'$ over $p$, and let $C_{s'}$ be the union of irreducible components of $\cX_{s'}$ corresponding to $C$. Up to shrinking $S$, we may assume that $C_{s'}$, as a reduced scheme, is the intersection of $\cX_{s'}$ and some union $\mathcal C$ of irreducible components of $\cX_{S'}$. Let $\mathcal I_{\mathcal C}$ be the sheaf of ideals on $\cX$ defining $\mathcal C$. Then the sheaf of ideals defining $C_{s'}$ is $\mathcal I_{\mathcal C}\otimes_{\mathcal O_{\cX_{S'}}}\mathcal O_{\cX_{s'}}$. Note that there are only finitely many possibilities for $\mathcal C$. 

Let $k$ be a positive integer such that $\mathcal L^{\otimes k}$ has a nonzero section. Up to shrinking $S$, we may assume that this section does not vanish along any component of a fiber of $\cX_{S'}\ra S'$. Consider the map
$$\pi : \cX_{S'} \ra S'$$
If $n$ is large enough and since $\mathcal L$ is relatively ample, relative vanishing guarantees that the coherent sheaf on $S$
$$R^1\pi_*(\mathcal L^{\otimes n}\otimes_{\mathcal O_{\cX_{S'}}}\mathcal I_{\mathcal C})$$
is zero. Pick a positive $N$ once and for all such that the vanishing above holds for $n=N, \ldots, N+k-1$. Then after shrinking $S$ once again, we may assume that the vanishing above implies 
$$H^1(\cX_{s'}, \mathcal L^{\otimes N+i}\otimes_{\mathcal O_{\cX_{S'}}}\mathcal I_{C'})=0$$
for $i=0, \ldots, k-1$. 

Now since $\mathcal L^{\otimes k}$ has a nonzero section over $\cX_{s'}$, we have an exact sequence, for any integer $n$,
$$0\ra \mathcal L^{\otimes n}\otimes_{\mathcal O_{\cX_{S'}}}\mathcal I_{C'}\ra \mathcal L^{\otimes n+k}\otimes_{\mathcal O_{\cX_{S'}}}\mathcal I_{C'}\ra \mathcal K\ra 0$$
where $\mathcal K$ is a coherent sheaf supported on a zero-dimensional subscheme of $\cX_{s'}$. In particular, the map 
$$H^1(\cX_{s'}, \mathcal L^{\otimes n}\otimes_{\mathcal O_{\cX_{S'}}}\mathcal I_{C'})\ra H^1(\cX_{s'}, \mathcal L^{\otimes n+k}\otimes_{\mathcal O_{\cX_{S'}}}\mathcal I_{C'})$$
is onto and the right-hand term vanishes as soon as the left-hand one does. Finally, we have found $N$, independent of $C$ and $p$, such that for all $n\geq N$, we have 
$$H^1(\cX_{s'}, \mathcal L^{\otimes n}\otimes_{\mathcal O_{\cX_{S'}}}\mathcal I_{C'})=0,$$
which implies that the map 
$$H^0(\cX_p, \mathcal L_p^{\otimes n})\ra H^0(C, \mathcal L_p^{\otimes n})$$
is surjective.
\end{proof}

\begin{proposition}\label{proposition:irreducible-finite-field}
Let $p$ be a prime number corresponding to a point in $S$, and let $\mathcal L_p$ be the restriction of $\mathcal L$ to $\cX_p$. Let $C$ be an irreducible component of $\cX_p$, and let $r_C$ be the number of irreducible components of $C_{\overline \Fp}$. 

Let $\beta$ be a real number with $0<\beta<1$. There exist positive constants $A$ and $B$, depending only on $\beta$ and $\cX\ra S$, such that for any $n\geq A$, the proportion of those sections $s$ of $H^0(\cX_p, \mathcal L_p^{\otimes n})$ that do not vanish identically on $C$ and such that $\mathrm{div}(s)$ has an irreducible component of degree at least $r_C(nd-n^\beta)$ lying on $C$ is at least $Bn^{\beta-1}.$
\end{proposition}

\begin{proof}
Our assumption on $p$ guarantees that $C$ is reduced. The degree of $\mathcal L_p$ on $C$ is equal to $r_Cd$. Let $n$ be a large enough positive integer. Let $k$ be an integer such that $nr_cd\geq r_Ck.$ Let $D$ be an irreducible divisor of degree $r_Ck$ on $C$. Then the number of sections of $\mathcal L_p^{\otimes n}$ over $C$ that vanish on $D$ is equal to the number of sections of $\mathcal L_p^{\otimes n}(-D)$ over $C$, which, according to Lemma \ref{lemma:Riemann-Roch}, is bounded below by 
$$p^{1-p_a(C)+nr_Cd-r_Ck}.$$

Assume that $r_Ck> \frac{1}{2}nr_Cd$. Then a nonzero section of $\mathcal L_p^{\otimes n}$ over $C$ vanishes on at most one irreducible divisor of degree $r_Ck$.  Applying Lemma \ref{lemma:Lang-Weil}, it follows that the number of nonzero sections of $\mathcal L_p^{\otimes n}$ over $C$ that vanish on some irreducible divisor of degree $r_Ck$ is bounded below by 
$$\frac{1}{k}p^{1-p_a(C)+nr_Cd}(1-O(p^{-\frac{1}{2}r_Ck}))-\frac{1}{k}p^{r_Ck},$$
the last term taking care of the zero section being counted multiple times.

Assume now that 
$$r_Ck\leq nr_Cd-p_a(C).$$
Then the term above is bounded below by 
$$\frac{1}{2k}p^{1-p_a(C)+nr_Cd}$$
for large enough $n$.

Summing over all those $k$ such that $r_Ck\geq nr_Cd-r_Cn^\beta$, we find that the number of those elements $s$ of $H^0(\cX_p, \mathcal L_p^{\otimes n})$ such that $\mathrm{div}(s)$ has an irreducible component of degree at least $nr_Cd-n^{\beta}$ is at least 
$$n^\beta\frac{1}{2nd}p^{1-p_a(C)+nr_Cd}(1+o(1)).$$
as $n$ goes to infinity, the implied constants depending only on $\beta$, $p_a(C)$ and the ones occurring in Lemma \ref{lemma:Lang-Weil}. Since $p_a(C)$ is the genus of some reunion of irreducible components of the geometric generic fiber of $\cX$, the implied constants only depend on $\beta$ and $\cX$.

By Lemma \ref{lemma:Riemann-Roch}, if $nr_Cd>p_a(C)$, we have 
$$h^0(C, \mathcal L_p^{\otimes n})=1-p_a(C)+nr_Cd.$$
This shows that the proportion of those sections $s$ of $\mathcal L_p^{\otimes n}$ over $C$ such that $\mathrm{div}(s)$ has an irreducible component of degree at least $nr_Cd-r_Cn^{\beta}$ is at least $Bn^{\beta-1}$ for some constant $B$ as in the statement of the proposition.

By Lemma \ref{lemma:vanishing-independent}, after choosing $n$ large enough, this implies the desired statement.
\end{proof}

\bigskip

We can now prove the main result of \ref{subsection:finite-fields}.

\begin{proposition}\label{proposition:geometric-irreducibility}
In the situation of \ref{subsection:setup}, let $\beta$ be a real number with $0<\beta<1$. Then the proportion of those elements of $H^0_{Ar}(\cX, \Lbar^{\otimes n})$ that vanish on an irreducible divisor of $\cX_\Q$ of degree at least $n\deg \mathcal L_\Q-rn^\beta$ goes to $1$ as $n$ goes to infinity.
\end{proposition}

\begin{proof}
Let $\gamma$ be a real number with $1-\beta<\gamma<1$. Let $n$ be a large enough integer. Letting $t$ be the smallest integer bigger than $n^\gamma$, let $p_1, \ldots, p_t$ be the $t$ smallest primes corresponding to points of $S$, and let $N$ be their product. By the prime number theorem, we have $p_t\sim t\log t$ as $t\ra\infty$, so that 
\begin{equation}\label{equation:bound-primes}
N\leq (t\log t)^t=O(e^{n^\gamma}),
\end{equation}
where $\gamma$ is any real number with $\gamma<\gamma<1.$

Write $\Lambda_n$ for $H^0(\cX, \mathcal L^{\otimes n})$ and let $\cX_N\ra \Spec\Z/N\Z$ be the reduction of $\cX$ modulo $N$. The exact sequence defining $\cX_N$ is
$$0\ra N\mathcal O_{\cX}\ra\mathcal O_{\cX}\ra \mathcal O_{\cX_N}\ra 0,$$
hence the exact sequence 
$$0\ra \Lambda_n/N\Lambda_n\ra H^0(\cX_N, \mathcal L^{\otimes n})\ra H^1(\cX, \mathcal L^{\otimes n})[N]\ra 0.$$
If $n$ is large enough, then $H^1(\cX, \mathcal L^{\otimes n})=0$ and we have 
\begin{equation}\label{equation:quotient}
H^0(\cX_N, \mathcal L^{\otimes n})=\Lambda_n/N\Lambda_n.
\end{equation}

The scheme $\cX_N$ is the disjoint union of the $\cX_{p_i}$, $1\leq i\leq t$. As a consequence, we have 
$$H^0(\cX_N, \mathcal L^{\otimes n})=\Pi_{1\leq i\leq t} H^0(\cX_{p_i}, \mathcal L^{\otimes n}).$$

Given a prime number $p$ that corresponds to a point of $S$, let $E_p$ be the subset of $H^0(\cX_{p}, \mathcal L^{\otimes n})$ described by Proposition \ref{proposition:irreducible-finite-field}: $E_p$ is the set of sections $s$ of $H^0(\cX_p, \mathcal L^{\otimes n})$ such that there exists an irreducible component $C$ of $\cX_p$, such that $C_{\overline\Fp}$ has $r_C$ irreducible components, the restriction of $s$ to $C$ is not identically zero and vanishes along an irreducible divisor $D_p$ of degree at least $r_C(nd-n^\beta).$ 

\bigskip

By Proposition \ref{proposition:irreducible-finite-field}, if $n$ is greater than $A$, the proportion of those elements $s$ of $H^0(\cX_N, \mathcal L^{\otimes n})$ such that $s$ does not project to $E_{p_i}$ for any $i\in\{1, \ldots, t\}$ is bounded above by 
$$(1-Bn^{\beta-1})^s\geq (1-Bn^{\beta-1})^{n^\gamma}=\mathrm{exp}(-Bn^{\gamma+\beta-1}+o(n^{\gamma+\beta-1}))=o(1)$$
since $\gamma+\beta-1>0$, so that as $n$ goes to infinity, the proportion of those elements of $H^0(\cX_N, \mathcal L^{\otimes n})$ that project to at least one of the $E_{p_i}$ goes to $1$.

By Proposition \ref{proposition:varying-quotient}, (\ref{equation:bound-primes}) and (\ref{equation:quotient}), the proportion of those elements of $H^0_{Ar}(\cX, \Lbar^{\otimes n})$ that restrict to $E_{p_i}$ for some $i\in\{1, \ldots, t\}$ goes to $1$ as $n$ goes to infinity. We claim that these elements satisfy the condition of the proposition we are proving.

\bigskip

Let $p$ be a prime number that corresponds to a point of $s$. Let $\sigma$ be a section of $\mathcal L^{\otimes n}$ over $\cX$ such that the restriction of $\sigma$ to $\cX_p$ belongs to $E_p$. Let $C$ be a component of $\cX_p$ such that $C_{\overline\Fp}$ has $r_C$ irreducible components, and let $D_p$ be an irreducible divisor of degree at least $r_C(nd-n^\beta)$ on $C$ such that $\sigma$ vanishes on $D_p$. 

We can find an irreducible component $D$ of $\mathrm{div}(\sigma)$ with $D_{|\cX_p}=D_p$. If $n$ is large enough, we can assume that no component of $(D_p)_{\overline \Fp}$ lies on two distinct irreducible components of $C_{\overline\Fp}$ -- it is enough to require that $n$ is large enough compared to the degree of the residue fields of the intersection points of any two components of $C_{\overline\Fp}$. As a consequence, the degree of the restriction of $D_p$ to any of the $r_C$ components of $C_{\overline \Fp}$ is at least $nd-n^\beta$. Since $\cX_{\Q}$ is irreducible, the degree of the restriction of $D_{\Q}$ to any component of $\cX_{\overline \Q}$ is at least $nd-n^\beta$ as well, so that the degree of $D_{\Q}$ is at least 
$$r(nd-n^{\beta})=n\deg\mathcal L_{\Q}-rn^\beta.$$
This is what we needed to prove.
\end{proof}

\bigskip

\subsection{End of the proof} We can finish the proof. The strategy follows roughly the outline of the proof of \cite[Proposition 4.1]{CharlesPoonen16} which deals with the corresponding result over finite fields.

Let $\pi : \widetilde{\cX}\ra \cX$ be a resolution of singularities of $\cX$. Recall that we denoted by $r$ the number of irreducible components of $\cX_\C$. Then the complex curve $\widetilde\cX_\C$ is the disjoint union of $r$ smooth, connected components.

Define $\Bbar=\pi^*\Lbar.$ Let $\omega$ be the first Chern class of $\Bbar$. Since the generic fiber of $\pi$ is finite, $\omega$ is positive except at a finite number of points of $\cX(\C)$. We say that a hermitian line bundle on $\widetilde\cX$ is admissible if it is admissible with respect to $\omega$, and we write $\widehat{\mathrm{Pic}}_\omega(\widetilde\cX)$ for the group of isomorphism classes of $\omega$-admissible hermitian line bundles on $\widetilde\cX$.

We have an exact sequence
$$0\ra \R\ra \widehat{\mathrm{Pic}}_\omega(\widetilde\cX)\ra \mathrm{Pic}(\widetilde\cX)\ra 0.$$
We fix once and for all a subgroup $N$ of $\widehat{\mathrm{Pic}}_\omega(\widetilde\cX)$ such that the following conditions hold:
\begin{enumerate}[(i)]
\item $N$ is a group of finite type;
\item $N$ surjects onto $\mathrm{Pic}(\widetilde\cX)$ and contains the class of $\Bbar$;
\item $N\cap\Ker(\widehat{\mathrm{Pic}}_\omega(\widetilde\cX)\ra \mathrm{Pic}(\widetilde\cX))$ has rank $1$.
\end{enumerate}
Such a group $N$ certainly exists since $\mathrm{Pic}(\widetilde\cX)$ is a group of finite type. Note that these mean that $N$ is a discrete cocompact subgroup of $\widehat{\mathrm{Pic}}_\omega(\widetilde\cX)$. In particular, there exists a positive constant $C$ such that for any hermitian line bundle $\Mbar=(\mathcal M, ||.||)$ on $\widetilde\cX$ with curvature form proportional to $\omega$, there exists a hermitian metric $||.||'$ on $\mathcal M$ such that $(\mathcal M, ||.||')$ belongs to $N$ and the norms $||.||$ and $||.||'$ satisfy the inequality
\begin{equation}\label{equation:perturbation-in-N}
C^{-1}||.||\leq ||.||\leq C||.||.
\end{equation}

\bigskip

The following result is classical in the geometric setting: big divisors are (rationally) the sum of ample divisors and effective divisors.

\begin{lemma}\label{lemma:big-is-ample-plus-effective}
The hermitian line bundle $\Bbar$ satisfies the conditions of Proposition \ref{proposition:upper-bound-theta}. Furthermore, there exists a positive integer $k$, and line bundles $\Abar$ and $\Ebar$ on $\widetilde\cX$ which are ample and effective respectively, such that 
$$\Bbar^{\otimes k}\simeq \Abar\otimes\Ebar.$$
\end{lemma}

\begin{proof}
Since $\Lbar$ is ample, some power of $\Lbar$ is effective, and so is the same power of $\Bbar$. We also have $\Bbar.\Bbar=\Lbar.\Lbar>0.$ Finally, let $\Mbar$ be an effective line bundle on $\widetilde\cX$, let $s$ be a nonzero effective section of $\Mbar$, and let $D$ be the divisor of $s$. Then 
$$\Bbar.\Mbar=h_\Bbar(D)-\int_{\widetilde{\cX}(\C)} \log ||s_\C||\pi^*c_1(\Lbar).$$
Considering an effective section of some power of $\Bbar$ that does not vanish along any component of $D$ -- which exists since large powers of $\Lbar$ are generated by their effective sections -- we see that first term is nonnegative. The second one is nonnegative as well since $c_1(\Lbar)$ is positive on $\cX$. This shows the first statement of the proposition. 

Let $\Abar$ be an effective ample line bundle on $\widetilde\cX$, let $\sigma$ be a section of $\Abar$ and let $H$ be the divisor of $\sigma$. Let $H'$ be the scheme $\pi(H)$. Since $\mathcal L$ is ample, we can find an integer $k_1$ and a nonzero section $s_1$ of $\mathcal L^{\otimes k_1}$ that vanishes on $H'$. We can write 
$$\pi^*s_1=\sigma\sigma_1,$$
where $\sigma_1$ is a section of $\mathcal B^{\otimes k_1}\otimes\mathcal A^{\otimes -1}.$ Choose an large enough integer $k_2$, and let $s_2$ be a nonzero section of $\mathcal L^{\otimes k_2}$ with small enough norm. Writing $\sigma_2=\pi^*s_2$, we have
$$\pi^*(s_1s_2)=\sigma\sigma_1\sigma_2,$$
and $\sigma_1\sigma_2$ is an effective section of the hermitian line bundle $\Bbar^{\otimes (k_1+k_2)}\otimes \Abar^{\otimes -1},$ which proves the result.
\end{proof}

\bigskip

Let $\alpha$ and $\beta$ be a real numbers with $0<\beta<\alpha<\frac{1}{2}$. If $n$ is a positive integer, let $Y_n$ be the subset of $H^0_{Ar}(\cX, \Lbar^{\otimes n})$ consisting of those effective sections $\sigma$ of $\Lbar^{\otimes n}$ such that:
\begin{enumerate}[(i)]
\item $\sigma$ does not vanish on any Weil divisor $D$ of $\cX$ with $h_{\Lbar}(D)\leq n^\alpha$;
\item there exists an irreducible component $D$ of $\mathrm{div}(\sigma)$ such that 
$$\mathrm{deg}(D_\Q)\geq n\mathrm{deg}\mathcal L_\Q-rn^\beta.$$
\end{enumerate}
Use Lemma \ref{lemma:big-is-ample-plus-effective} to find a positive integer $k$ with 
$$\Bbar^{\otimes k}\simeq \Abar\otimes\Ebar,$$
where $\Abar$ is ample and $\Ebar$ is effective. 

\begin{lemma}\label{lemma:properties-of-decomposition}
Let $\delta, \gamma$ be any real numbers with $0<\beta<\gamma<\delta<\alpha$. Let $n$ be a large enough integer, and let $\sigma$ be an element of $Y_n$ such that $\mathrm{div}(\sigma)$ is not irreducible. Then we can find hermitian line bundles $\Lbar_1$ and $\Lbar_2$ on $\widetilde\cX$, and sections 
$$\sigma_i\in H^0(\widetilde\cX, \mathcal L_i),$$
$i=1, 2$, with the following properties:
\begin{enumerate}[(i)]
\item $\Lbar_1$ and $\Lbar_2$ belong to $N$;
\item $||\sigma_i||\leq e^{n^\gamma}, i=1,2$;
\item $n^\delta\leq \Lbar_1.\Bbar\leq n\Bbar.\Bbar-n^\delta$
\item $\Lbar_1.\Abar\leq n\Bbar.\Bbar$.
\item $\Lbar_1\otimes\Lbar_2\simeq \Bbar^{\otimes n}$;
\item up to the isomorphism above, $\sigma_1\sigma_2=\sigma$;
\end{enumerate}
\end{lemma}

\begin{proof}
Let $D$ be the divisor of $\pi^*\sigma$. Since the divisor of $\sigma$ is not irreducible, $D$ is not irreducible either and we can write 
$$D=D_1+D_2$$
where the $D_i$ are nonzero effective divisors on the regular scheme $\widetilde\cX$ such that both Weil divisors $\pi(D_1)$ and $\pi(D_2)$ are nonzero. Since $\mathrm{div}(\sigma)$ has an irreducible component of generic degree bounded below by $n\deg\mathcal L_\Q -rn^\beta$, and since 
\begin{equation}\label{equation:sum-degrees}
\deg D_{1, \Q}+\deg D_{2, \Q}=n \deg \mathcal L_\Q,
\end{equation}
we can assume, up to exchanging $D_1$ and $D_2$,
\begin{equation}\label{equation:degree-1}
n\deg\mathcal L_\Q -rn^\beta\leq \deg D_{1, \Q}\leq n\deg\mathcal L_\Q,
\end{equation}
\begin{equation}\label{equation:degree-2}
0\leq \deg D_{2, \Q}\leq rn^\beta.
\end{equation}
We can also assume that no component of $D_1$ is contracted by the morphism $\pi : \widetilde\cX\ra \cX$ -- simply by replacing $D_2$ by the sum of $D_2$ and all those contracted components of $D_1$, which are all supported above closed points of $\Spec\Z$. Let $\mathcal L_i$ be the line bundle $\mathcal O_{\widetilde\cX}(D_i)$ for $i=1, 2$. Then we can identify $\mathcal L_1\otimes\mathcal L_2$ with $\mathcal B^{\otimes n}$, and we can find sections $\sigma_i$ of $\mathcal L_i$ with $\mathrm{div}(\sigma_i)=D_i$ such that $\sigma=\sigma_1\sigma_2$.

Recall that we defined $\omega$ as $c_1(\Bbar)$. We consider the norms $||.||_0$ with respect to $\omega$. Consider the unique hermitian metric $||.||_{\sigma_2}$ on $\mathcal L_2$ which is admissible with respect to $\omega$, scaled so that 
$$||\sigma_2||_{\sigma_2,0}=1.$$
By (\ref{equation:perturbation-in-N}), we can find a metric $||.||$ on $\mathcal L_2$ such that $\Lbar_2:=(\mathcal L_2, ||.||)$ belongs to $N$ and 
\begin{equation}\label{equation:0-1}
C^{-1}\leq ||\sigma_2||_{0}\leq C.
\end{equation}
Endow $\mathcal L_1$ with the unique hermitian metric such that 
$$\Bbar^{\otimes n}=\Lbar_1\otimes\Lbar_2$$
as hermitian line bundles on $\widetilde\cX$, where we write $\Lbar_1$ for the induced hermitian line bundles. Since $\Bbar$ belongs to $N$ by assumption, so do $\Lbar_1$ and $\Lbar_2$. This makes sure that conditions $(i)$, $(v)$ and $(vi)$ of the lemma are satisfied.

Since $||\sigma||_\infty\leq 1$, we have 
\begin{equation}\label{equation:0-2}
||\sigma_2||_0\leq C||\sigma_1|||_0||\sigma_2||_0=C||\sigma||_0\leq C.
\end{equation}

\bigskip

The inequalities (\ref{equation:sum-degrees}), (\ref{equation:degree-2}) imply, via Proposition \ref{proposition:bound-large-degree} the following estimate, since $||\sigma_2||_0\geq C^{-1}$ and $||\sigma||_\infty\leq 1$:
$$||\sigma_1||_{\infty}\leq C^{-1}(nC_2\deg\mathcal L_\Q)^{rn\beta}$$
for some constant $C_2$ depending only on $\cX$ and $\Lbar$. Similarly, (\ref{equation:degree-2}) and Proposition \ref{proposition:bound-low-degree} give us, for some constant $C_2$ depending only on $\cX$ and $\Lbar$:
$$||\sigma_2||_\infty\leq CC_1^{rn^\beta}.$$
For any $\gamma>\beta$, and any $n$ large enough, this ensures that condition $(ii)$ is satisfied.

\bigskip

We now turn to condition $(iii)$. For $i=1, 2$, choose a nonzero effective section $s_i$ of some power $\Lbar^{\otimes k}$ of $\Lbar$ such that the divisor of $\pi^*s_i$ has no common component with $D_i$. Writing
$$k\Lbar_i.\Bbar=h_{\Bbar^{\otimes k}}(D_i)-k\int_{\cX(\C)}\log||\sigma_i||\omega$$
and computing $h_{\Bbar^{\otimes k}}(D_i)$ using $\pi^*s_i$, we see that 
$$h_{\Bbar^{\otimes k}}(D_i)= h_{\Lbar^{\otimes k}}(\pi(D_i))\geq n^\alpha.$$
As a consequence, we find 
\begin{equation}\label{equation:Lbar-Bbar}
\Lbar_i.\Bbar\geq n^\alpha-\deg\mathcal L_\Q n^\gamma\geq n^\delta
\end{equation}
for any large enough $n$ since $\delta, \gamma<\alpha$. Since 
$$\Lbar_1.\Bbar+\Lbar_2=n\Bbar.\Bbar,$$
this proves that $(iii)$ holds.

Let us prove condition $(iv)$. The first inequality holds by (\ref{equation:Lbar-Bbar}). Certainly, we have
$$\Lbar_i.\Abar=\Lbar_i.\Bbar-\Lbar_i.\Ebar$$
for $i=1,2$, so that 
\begin{equation}\label{equation:degree-of-L1}
\Lbar_1.\Abar=n\Bbar.\Bbar-\Lbar_2.\Bbar-\Lbar_1.\Ebar.
\end{equation}

Let $\tau$ be a nonzero effective section of $\Ebar$, with divisor $D_\tau$. Then we have
$$\Lbar_1.\Ebar=h_{\Lbar_1}(D_\tau)-\int_{\widetilde\cX(\C)}\log ||\tau|| c_1(\Lbar_1).$$
Since the degree of $\Lbar_1$ is nonnegative, the form $c_1(\Lbar_1)$ is a nonnegative multiple of $\omega$, and since $\tau$ is effective, we have
$$-\int_{\widetilde\cX(\C)}\log ||\tau|| c_1(\Lbar_1)\geq 0.$$
By assumption, no component of the divisor $D_1$ of $\sigma_1$ is contracted by the resolution $\pi$. Furthermore, the definition of the set $Y_n$ guarantees that if $C$ is any component of $D_1$, then the height of $\pi(C)$ with respect to $\Lbar$ is bounded below by $n^\alpha$. This implies that if $n$ is large enough, the divisors $D_1$ and $D_\tau$ have no component in common, so that 
$$h_{\Lbar_1}(D_\tau)\geq -\deg D_{\tau, \Q} \log ||\sigma_1||\geq -\deg D_{\tau, \Q} n^\gamma$$
and, as a consequence,
\begin{equation}\label{equation:Lbar-Ebar}
\Lbar_1.\Ebar\geq -\deg \mathcal E_\Q n^\gamma.
\end{equation}

Putting the inequalities (\ref{equation:Lbar-Ebar}) and (\ref{equation:Lbar-Bbar}) together with (\ref{equation:degree-of-L1}), we obtain
$$\Lbar_1.\Abar\leq n\Bbar.\Bbar+(\deg \mathcal E_\Q)n^\gamma-n^\delta.$$
Since $\Lbar$ is ample, $\Bbar.\Bbar=\Lbar.\Lbar$ is positive, and since $\gamma<\delta$, this shows that condition $(iv)$ of the lemma is satisfied as soon as $n$ is large enough.
\end{proof}

\bigskip

We can finally prove the key result of this paper via a counting argument.

\begin{proposition}\label{proposition:arithmetic-surface}
Let $\cX$ be an integral projective arithmetic surface, and let $\Lbar$ be an ample hermitian line bundle on $\cX$. Then the set 
$$\{\sigma\in\bigcup_{n> 0} H^0_{Ar}(\cX, \Lbar^{\otimes n}), \,\mathrm{div}(\sigma)\,\textup{is irreducible}\}$$
has density $1$.
\end{proposition}

\begin{proof}
Choose $\delta$ and $\gamma$ with $\beta<\gamma<\delta<\alpha$. From Proposition \ref{proposition:divisors-with-low-height} and Proposition \ref{proposition:geometric-irreducibility}, we know that the set $\bigcup_{n>0} Y_n$ has density $1$ in $H^0_{Ar}(\cX, \Lbar^{\otimes n})$, so that we only have to prove that the set of those $\sigma$ in $\bigcup_{n>0} Y_n$ with reducible divisor has density $0$ in $H^0_{Ar}(\cX, \Lbar^{\otimes n})$. Let $Z_n$ be this set.

Let $n$ be large enough so that Lemma \ref{lemma:properties-of-decomposition} applies. To any $\sigma$ in $Z_n$, we can associate hermitian line bundles $\Lbar_1$ and $\Lbar_2$, together with respective sections $\sigma_1$ and $\sigma_2$, so that the conditions $(i)-(vi)$ of the lemma hold. Since $\sigma=\sigma_1\sigma_2$, the data of the $\Lbar_i$ and $\sigma_i$ for $i=1, 2$ determine $\sigma$. 

We will give an upper bound for the number of elements $\sigma$ in $Z_n$ by estimating the number of possible $\Lbar_i$ and $\sigma_i$. In other words, we will count the number of triples $(\Lbar_1, \sigma_1, \sigma_2)$, where $\Lbar_1$ is a hermitian line bundle on $\widetilde\cX$, $\sigma_1$ is a section of $\mathcal L_1$, and, setting $\Lbar_2:=\Bbar^{\otimes n}\otimes \Lbar_1^{\otimes -1}$, $\sigma_2$ is a section of $\Lbar_2$, so that 
\begin{enumerate}[(i)]
\item $\Lbar_1$ and $\Lbar_2$ belong to $N$;
\item $||\sigma_i||\leq e^{n^\gamma}, i=1,2$;
\item $n^\delta\leq \Lbar_1.\Bbar \leq n\Bbar.\Bbar-n^\delta$;
\item $\Lbar_1.\Abar\leq n\Bbar.\Bbar$.
\end{enumerate}
Below, when using the $O$ notations, implied constants only depend on $\widetilde\cX\ra\cX, \Lbar, \Abar, \alpha, \beta, \delta, \gamma$.

\bigskip

Let $\Lbar_1$ be a hermitian line bundle as above, and write $i:=\Lbar_1.\Bbar$, so that 
$$n^\delta\leq i\leq n\Bbar.\Bbar-n^\delta.$$ 
We want to bound the number of sections of $\mathcal L_1$ that have norm at most $e^{n^\gamma}$, that is, the number of effective sections of $\Lbar_1(n^\gamma)$. First remark that $\deg \mathcal L_{1, \Q}\leq n\deg \mathcal B_\Q$ as the degree of $\mathcal L_{1, \Q}$ and $\mathcal L_{2, \Q}$ are both nonnegative and have sum $n\deg \mathcal B_\Q$. Furthermore, we have 
$$\Lbar_1(n^\gamma).\Bbar=i+n^\gamma\Obar_{\widetilde\cX}(1).\Bbar=O(n)$$
since $\gamma<\delta<1.$ 

Corollary \ref{corollary:upper-bound-arakelov} gives us
\begin{equation}\label{equation:number-of-sigma-1}
h^0_{Ar}(\widetilde\cX, \Lbar_1(n^\gamma))\leq \frac{(i+Kn^\gamma)^2}{2\Bbar.\Bbar}+O(n\log n),
\end{equation}
where $K$ is the constant $\Obar_{\widetilde\cX}(1).\Bbar$.

Similarly, we have 
\begin{equation}\label{equation:number-of-sigma-2}
h^0_{Ar}(\widetilde\cX, \Lbar_2(n^\gamma))\leq \frac{(n\Bbar.\Bbar-i+Kn^\gamma)^2}{2\Bbar.\Bbar}+O(n\log n).
\end{equation}

Adding (\ref{equation:number-of-sigma-1}) and (\ref{equation:number-of-sigma-2}), we find, recalling that $0<\gamma<1$:
\begin{align*}
h^0_{Ar}(\widetilde\cX, \Lbar_1(n^\gamma))+h^0_{Ar}(\widetilde\cX, \Lbar_2(n^\gamma)) & \leq \frac{1}{2}n^2\Bbar.\Bbar-\frac{2i(n\Bbar.\Bbar-i)}{2\Bbar.\Bbar}+\frac{2K^2n^{2\gamma}+2K\Bbar.\Bbar n^{1+\gamma}}{2\Bbar.\Bbar}+O(n\log n)\\
& \leq \frac{1}{2}n^2\Bbar.\Bbar-\frac{i(n\Bbar.\Bbar-i)}{\Bbar.\Bbar}+O(n^{1+\gamma}).
\end{align*}
Since $n^\delta\leq i\leq n\Bbar.\Bbar-n^\delta$, we have
$$\frac{i(n\Bbar.\Bbar-i)}{\Bbar.\Bbar}\geq n^{1+\delta}-\frac{1}{\Bbar.\Bbar}n^{2\delta}$$
and, since $2\delta<1<1+\gamma$,
$$h^0_{Ar}(\widetilde\cX, \Lbar_1(n^\gamma))+h^0_{Ar}(\widetilde\cX, \Lbar_2(n^\gamma))\leq \frac{1}{2}n^2\Bbar.\Bbar-n^{1+\delta}+O(n^{1+\gamma}).$$

\bigskip

We now count the number of possible $\Lbar_1$. Let $t>0$ be such that $\Obar(t)$ belong to $N$. Let $k(n)$ be the smallest positive integer such that $k(n)t\geq n^\gamma$. Then the hermitian line bundle $\Lbar_1(k(n)t)$ is effective, belongs to $N$, and we have 
$$\Lbar_1(k(n)t).\Abar=O(n)$$
since $\gamma<1$.  As a consequence of Proposition \ref{proposition:cone}, this shows that the number of possible $\Lbar_1$ -- or equivalently, $\Lbar_1(k(n)t)$ -- appearing in the triples above is bounded by $O(n^\rho)$, where $\rho$ is the rank of $N$. 

\bigskip

The estimates above show that we have the following inequality:
$$\log |Z_n|\leq O(\rho \log n)+\frac{1}{2}n^2\Bbar.\Bbar-n^{1+\delta}+O(n^{1+\gamma})=\frac{1}{2}n^2\Bbar.\Bbar-n^{1+\delta}+O(n^{1+\gamma}).$$
However, Theorem \ref{theorem:arithmetic-Hilbert-Samuel}, (iii) shows that we have
$$h^0_{Ar}(\cX, \Lbar^{\otimes n})\geq\frac{1}{2}n^2\Lbar.\Lbar+O(n\log n)=\frac{1}{2}n^2\Bbar.\Bbar+O(n\log n).$$
Since $\delta>\gamma$, these two inequalities prove that $\bigcup_{n>0} Z_n$ has density $0$ in $H^0_{Ar}(\cX, \Lbar^{\otimes n})$, which proves the proposition.

\end{proof}

\section{Proofs of the main results}

The goal of this section is to give a proof of Theorem \ref{thm:main}. We will deduce it from its special case Theorem \ref{thm:X=Y}

\subsection{Proof of Theorem \ref{thm:X=Y}}.

We first state the Bertini irreducibility theorem of \cite{CharlesPoonen16} in the form that we will need. 

\begin{theorem}\label{theorem:finite-fields}
Let $k$ be a finite field, and let $X$ be a projective variety over $k$. Let $L$ be an ample line bundle over $X$. Let $Y$ be an integral scheme of finite type over $k$, and let $f : Y\ra X$ be a morphism which is generically smooth onto its image. Assume that the dimension of the closure of $f(Y)$ is at least $2$. Then the set of those $\sigma\in\bigcup_{n>0}H^0(X, \mathcal L^{\otimes n})$ such that $\mathrm{div}(f^*\sigma)_{horiz}$ is an irreducible Cartier divisor has density $1$. 
\end{theorem}

\begin{proof}
This is almost a special case of \cite[Theorem 1.6]{CharlesPoonen16}. There, the result is given when $X$ is a projective space and $L=\mathcal O(1)$. This means that -- unfortunately -- \cite{CharlesPoonen16} can formally only be applied to the situation where $L$ is very ample. However, the proofs of \cite{CharlesPoonen16} apply with no change when projective space is replaced by an arbitrary projective scheme with a distinguished ample line bundle. 

A second difference between our statement and that of \cite[Corollary 1.4]{CharlesPoonen16} is that we claim that we can require $\mathrm{div}(f^*\sigma)$ to be irreducible as a Cartier divisor: the underlying scheme is irreducible and has no multiple component, whereas the statement in \cite{CharlesPoonen16} only states irreducibility. 

The fact that for a density $1$ of $\sigma$, the divisor $\mathrm{div}(\sigma)$ has no multiple component follows from arguments in \cite{CharlesPoonen16}. Indeed, since $Y$ is reduced and $k$ is perfect, there is a dense open subset $U$ of $Y$ that is smooth over $k$ and such that $f_{|U}$ is smooth onto its image. By \cite[Lemma 3.3]{CharlesPoonen16}, for a density $1$ of sections $\sigma$, all the components of $\mathrm{div}(f^*\sigma)_{horiz}$ intersect $U$, and by \cite[Lemma 3.5]{CharlesPoonen16}, for a density $1$ of $\sigma$, the intersection $\mathrm{div}(f^*\sigma)\cap U$ is smooth outside a finite number of points, so that it does not have any multiple component.
\end{proof}

\begin{lemma}\label{lemma:no-vertical-component}
Let $\cX$ be a projective arithmetic variety of dimension at least $2$, and let $\Lbar$ be an ample line bundle on $\cX$. Then the set 
$$\{\sigma\in \bigcup_{n>0} H^0_{Ar}(\cX, \Lbar^{\otimes n}), \mathrm{div}(\sigma)\, \textup{has no vertical component}\}$$
has density $1$.
\end{lemma}

\begin{proof}
If $\cX$ is an arithmetic surface, the result follows from Proposition \ref{proposition:arithmetic-surface}. Let $d$ be the relative dimension of $\cX$ over $\Spec\Z$, and assume that $d\geq 2$.

Apply Theorem \ref{theorem:independent-restriction} where $Y$ runs through the irreducible components of the fibers of $\cX$ over closed points of $\Spec\Z$. Since these components have dimension $d$, we find that for any small enough $\varepsilon>0$, the proportion of these elements $\sigma$ of $H^0_{Ar}(\cX, \Lbar^{\otimes n})$ such that $\mathrm{div}(\sigma)$ has a vertical component over some prime $p$ with $p\leq \mathrm{exp}(\varepsilon n^2)$ is bounded above by a quantity of the form
$$O(\mathrm{exp}(\varepsilon n^2-\eta n^{d}))=o(1)$$
as $n$ goes to infinity.

\bigskip

We now show that for most $\sigma\in H^0_{Ar}(\cX, \Lbar^{\otimes n})$, $\mathrm{div}(\sigma)$ does not have any vertical component above a large prime.

Let $C\subset \cX$ be a closed arithmetic curve, flat over $\Spec\Z$, such that for any large enough prime $p$, the intersection of $C$ with any irreducible component of the fiber $\cX_p$ of $\cX$ above $p$ is nonempty. Let $n$ be a positive integer, and let $\sigma$ be an element of $H^0_{Ar}(\cX, \Lbar^{\otimes n})$. If $\mathrm{div}(\sigma)$ does not contain $C$, and if it has a vertical component above a prime $p$, then $\mathrm{div}(\sigma)$ and $C$ intersect at a point above $p$, so that 
$$nh_\Lbar(C) = h_{\Lbar^{\otimes n}}(C)\geq \log p.$$
In particular, for such a $\sigma$, we have $p\leq \mathrm{exp}(nh_\Lbar(C)).$

By Theorem \ref{theorem:independent-restriction}, the proportion of those $\sigma\in H^0_{Ar}(\cX, \Lbar^{\otimes n})$ that vanish on $C$ tends to $0$ as $n$ tends to infinity. In particular, the proportion of those $\sigma\in H^0_{Ar}(\cX, \Lbar^{\otimes n})$ such that $\mathrm{div}(\sigma)$ has a vertical component above a prime $p>\mathrm{exp}(nh_\Lbar(C))$ goes to $0$ as $n$ goes to infinity.

Together with the above estimate, this shows the result.
\end{proof}

\begin{proof}[Proof of Theorem \ref{thm:X=Y}]
If $\cX$ is an arithmetic surface, then the result was proved in Proposition \ref{proposition:arithmetic-surface}. Assume that $\cX$ has dimension at least $3$. Let $p$ be a prime number large enough so that $\cX_p$ is reduced, and specialization indices a bijection between the irreducible components of $\cX_{\overline \Q}$ and those of $\cX_{\overline{\mathbb F_p}}$. Let $\cX_{0, p}$ be an irreducible component of $\cX_p$, endowed with the reduced structure. 


Let $n$ be a positive integer, and let $\sigma$ be a  global section of $\mathcal L^{\otimes n}$. If $D$ is a horizontal component of $\mathrm{div}(\sigma)$, then $D$ intersects all components of $\cX_{\overline \Q}$, so that $D$ intersects $\cX_{0, p}$. This shows that for any section $\sigma$ of $\mathcal L^{\otimes n}$, if $\mathrm{div}(\sigma_{|\cX_{0,p}})$ is irreducible as a Weil divisor, then $\mathrm{div}(\sigma)$ has a single component that is flat over $\Z$. 

Now we have the following results:
\begin{enumerate}[(i)]
\item the density of those $\sigma_p\in\bigcup_{n>0}H^0(\cX_{0,p}, \mathcal L^{\otimes n})$ such that $\mathrm{div}(\sigma_p)$ is an irreducible Cartier divisor is $1$;
\item the density of those $\sigma\in \bigcup_{n>0} H^0_{Ar}(\cX, \Lbar^{\otimes n})$ such that $\mathrm{div}(\sigma)$ does not have a vertical component is $1$.
\end{enumerate}
Indeed, (i) follows from Theorem \ref{theorem:finite-fields} with $X=Y$, and (ii) is Lemma \ref{lemma:no-vertical-component}. By the discussion above, if $\sigma$ satisfies (i) and (ii), then $\mathrm{div}(\sigma)$ is irreducible. Finally, Corollary \ref{corollary:density-and-restriction} shows that the density of those $\sigma\in \bigcup_{n>0} H^0_{Ar}(\cX, \Lbar^{\otimes n})$ such that the restriction of $\sigma$ to $\cX_{0,p}$ satisfies (i) is $1$. This proves the result.
\end{proof}

\subsection{Proof of Theorem \ref{thm:main}}

In this section, we deduce Theorem \ref{thm:main-dominant} from Theorem \ref{thm:X=Y}, following the arguments of \cite[Section 5]{CharlesPoonen16}. We then prove Theorem \ref{thm:main} as a consequence.

In the following, fix a projective arithmetic variety $\cX$, together with an ample hermitian line bundle $\Lbar$.

\begin{lemma}\label{lemma:open-dense}
Let $\cY$ be an irreducible scheme of finite type over $\Spec\Z$, together with a morphism $f : \cY\ra\cX$. Let $U$ be an open dense subscheme of $\cY$. Then for all $\sigma$ in a density $1$ subset of $\bigcup_{n>0} H^0_{Ar}(\cX, \Lbar^{\otimes n})$, we have the equivalence
$$\mathrm{div}(f^*\sigma)_{horiz}\,\textup{is irreducible}\Leftrightarrow (\mathrm{div}(f^*\sigma)\cap U)_{horiz}\,\textup{is irreducible}.$$
\end{lemma}

\begin{proof}
This analogous to \cite[Lemma 3.3]{CharlesPoonen16}. The implication 
$$\mathrm{div}(f^*\sigma)_{horiz}\,\textup{is irreducible}\implies (\mathrm{div}(f^*\sigma)\cap U)_{horiz}\,\textup{is irreducible}$$
always holds. We prove the reverse implication.

Let $D$ be an irreducible component of $\cY\setminus U$ whose image under $f$ is positive-dimensional -- meaning by definition that $D$ is a component of $(\cY\setminus U)_{horiz}$. By Theorem \ref{theorem:independent-restriction}, the density of those $\sigma\in\bigcup_{n>0} H^0_{Ar}(\cX, \Lbar^{\otimes n})$ that vanish identically on $f(D)$ is zero. 

Now assume that $\sigma$ does not vanish identically along any component of $(\cY\setminus U)_{horiz}$ -- this is a condition satisfied by a density $1$ set of sections by the paragraph above. Then any horizontal component of $\mathrm{div}(f^*\sigma)_{horiz}$ meets $U$, which implies that the Zariski closure of $(\mathrm{div}(f^*\sigma)\cap U)_{horiz}$ is $\mathrm{div}(f^*\sigma)_{horiz}$.

In particular, for those $\sigma$, the implication 
$$(\mathrm{div}(f^*\sigma)\cap U)_{horiz}\,\textup{is irreducible}\implies \mathrm{div}(f^*\sigma)_{horiz}\,\textup{is irreducible}$$
holds.
\end{proof}

\begin{lemma}\label{lemma:etale-maps}
Let $\cY$ and $\cZ$ be two irreducible schemes that are flat, of finite type over $\Spec\Z$. Let 
$$\pi : \cY\ra \cZ$$
be a finite étale morphism, and let 
$$\psi : \cZ\ra \cX$$
be a morphism that has relative dimension $s$ at all points of $\cZ$. Assume that the dimension of the closure of $\psi(\cZ)$ in $\cX$ is at least $2$. Then for all $\sigma$ in a density $1$ subset of $\bigcup_{n>0} H^0_{Ar}(\cX, \Lbar^{\otimes n})$, we have the implication 
$$\mathrm{div}(\psi^*\sigma)\,\textup{is irreducible} \implies \mathrm{div}(\pi^*\psi^*\sigma)\,\textup{is irreducible}.$$
\end{lemma}

\begin{proof}
We follow the argument of \cite[Lemma 5.1]{CharlesPoonen16}. Irreducibility is more difficult to achieve if we replace $\cY$ by a finite cover. As a consequence, we may assume that $\pi$ is a Galois étale cover. Let $G$ be the corresponding Galois group. Let $m$ be the dimension of $\overline{\psi(\cZ)}$.

\bigskip

If $z$ is a closed point of $\cZ$, let $|z|$ be the cardinality of the residue field of $z$ and let $F_z$ denote the conjugacy class in $G$ associated to the Frobenius. We claim that for a density $1$ set of $\sigma$, the conjugacy classes $F_z$ cover all conjugacy classes of $G$ as $z$ runs through the closed points of $\mathrm{div}(\psi^*\sigma)$.

Indeed, let $C$ be such a conjugacy class. Let $U$ be a normal, dense affine open subset of $\cZ$. By the Chebotarev density theorem of \cite[Theorem 9.11]{Serre12} applied to $\pi^{-1}(U)\ra U$, the number of closed points $z$ of $U$ with $|z|\leq t$ and $F_z=C$ is equivalent to 
$$\frac{|C|}{|G|}\frac{t^{s+m}}{(s+m)\log t}$$
as $t$ tends to $\infty$. Let $E_{C, t}$ be the set of those $z$. 

By the Lang-Weil estimates, since the fibers of $\psi$ have all dimension $s$, the number of points $z$ with $|z|\leq t$ in a given fiber of $\psi$ above a closed point is bounded above by a quantity of the form
$$\alpha t^s,$$
for some positive $\alpha$, so that $|\psi(E_{C, t})|$ is bounded below by a quantity of the form
$$\beta\frac{t^m}{\log t}$$
for some positive $\beta$. Note that if $x\in \psi(E_{C, t})$, then $|x|\leq t$.  

Fix $t$ large enough. Theorem \ref{theorem:uniform-surjectivity} shows that the density of those $\sigma\in \bigcup_{n>0} H^0_{Ar}(\cX, \Lbar^{\otimes n})$ that do not vanish on any element of $\psi(E_{C, t})$ is equal to 
$$\Pi_{x\in \psi(E_{C, t})} (1-|x|^{-1})\leq (1-t^{-1})^{\beta t^m/\log t}=\exp\Big(-\beta \frac{t^{m-1}}{\log t}(1+o(1))\Big),$$
which tends to zero as $t$ tends to $\infty$ since $m\geq 2$. As a consequence, the density of those $\sigma$ such that $\psi^*\sigma$ vanishes at a closed point $z$ with $F_z=C$ is $1$, which proves the claim.

\bigskip

Now let $\sigma\in\bigcup_{n>0} H^0_{Ar}(\cX, \Lbar^{\otimes n})$ such that $\mathrm{div}(\psi^*\sigma)$ is irreducible and contains closed points $z$ such that the $F_z$ cover all conjugacy classes of $G$. Then $\pi^{-1}\mathrm{div}(\psi^*\sigma)=\mathrm{div}(\pi^*\psi^*\sigma)$ is irreducible. This proves the lemma.
\end{proof}

\begin{proof}[Proof of Theorem \ref{thm:main-dominant}]
We follow the argument of \cite[Lemma 5.2]{CharlesPoonen16}. By Lemma \ref{lemma:open-dense}, we can replace $\cY$ by any dense open subscheme. As a consequence, we can assume that $f$ factors as 
\[
\xymatrix{
\cY \ar[r]^{\pi} & \cZ \ar[r]^{\psi} & \cX
}
\]
where $\pi$ is finite étale, $\cZ$ is an open subset of some affine space $\mathbb A^s_\cX$ and $\psi$ is the projection onto $\cX$ -- indeed, the function field of $\cY$ is a finite separable extension of a purely transcendental extension of the function field of $\cX$.

By Lemma \ref{lemma:open-dense} and Lemma \ref{lemma:etale-maps}, for $\sigma$ in a density $1$ subset of $\bigcup_{n>0} H^0_{Ar}(\cX, \Lbar^{\otimes n})$, the implication 
$$\mathrm{div}(\sigma)\,\textup{is irreducible} \implies \mathrm{div}(f^*\sigma)_{horiz}\,\textup{is irreducible}$$
holds. By Theorem \ref{thm:X=Y}, the divisor $\mathrm{div}(\sigma)$ is irreducible for $\sigma$ in a density $1$ subset of $\bigcup_{n>0} H^0_{Ar}(\cX, \Lbar^{\otimes n})$, which proves the result.
\end{proof}

\begin{proof}[Proof of Theorem \ref{thm:main}]
We first assume that $\cY$ is not flat over $\Spec\Z$. Then $f : \cY\ra \cX$ factors as
\[
\xymatrix{
\cY \ar[r]^{f_p} & \cX_p \ar[r] & \cX
}
\]
for some prime number $p$. By Theorem \ref{theorem:finite-fields}, the density of those $s\in \bigcup_{n>0} H^0(\cX_p, \mathcal L^{\otimes n})$ such that $\mathrm{div}(f_p^*s)_{horiz}$ is irreducible is equal to $1$. Applying Corollary \ref{corollary:density-and-restriction} to $\Lbar(\varepsilon)$ proves the theorem.

\bigskip

We now assume that $\cY$ is flat over $\Spec\Z$. Let $\cY'$ be the Zariski closure of $f(\cY)$ in $\cX$. Then $\cY'$ is a projective arithmetic variety, and the restriction of $\Lbar$ to $\cY'$ is ample by Corollary \ref{corollary:pull-back-ample}. Furthermore, the map $f_{\cY'} : \cY\ra \cY'$ is dominant by assumption. Theorem \ref{thm:main-dominant} guarantees that the density of those $\sigma\in\bigcup_{n>0} H^0_{Ar}(\cY', \Lbar^{\otimes n})$ such that $\mathrm{div}(f_{\cY'}^*\sigma)_{horiz}$ is irreducible is equal to $1$. Applying Corollary \ref{corollary:density-and-restriction} to $\Lbar(\varepsilon)$ proves the theorem.
\end{proof}

\bibliographystyle{alpha}
\bibliography{Bertini}

\begin{thebibliography}{{Sta}17}

\bibitem[AB95]{AbbesBouche95}
A.~Abbes and T.~Bouche.
\newblock Th\'eor\`eme de {H}ilbert-{S}amuel ``arithm\'etique''.
\newblock {\em Ann. Inst. Fourier (Grenoble)}, 45(2):375--401, 1995.

\bibitem[AT69]{AndreottiTomassini69}
A.~Andreotti and G.~Tomassini.
\newblock A remark on the vanishing of certain cohomology groups.
\newblock {\em Compositio Math.}, 21:417--430, 1969.

\bibitem[Aut01]{Autissier01}
Pascal Autissier.
\newblock Points entiers et th\'eor\`emes de {B}ertini arithm\'etiques.
\newblock {\em Ann. Inst. Fourier (Grenoble)}, 51(6):1507--1523, 2001.

\bibitem[Aut02]{Autissier02}
Pascal Autissier.
\newblock Corrigendum: ``{I}nteger points and arithmetical {B}ertini theorems''
  ({F}rench).
\newblock {\em Ann. Inst. Fourier (Grenoble)}, 52(1):303--304, 2002.

\bibitem[BGS94]{BostGilletSoule94}
J.-B. Bost, H.~Gillet, and C.~Soul{\'e}.
\newblock Heights of projective varieties and positive {G}reen forms.
\newblock {\em J. Amer. Math. Soc.}, 7(4):903--1027, 1994.

\bibitem[Blo81]{Bloch81}
Spencer Bloch.
\newblock Algebraic {$K$}-theory and classfield theory for arithmetic surfaces.
\newblock {\em Ann. of Math. (2)}, 114(2):229--265, 1981.

\bibitem[Bos91]{Bost91}
Jean-Beno{\^{\i}}t Bost.
\newblock Th\'eorie de l'intersection et th\'eor\`eme de {R}iemann-{R}och
  arithm\'etiques.
\newblock {\em Ast\'erisque}, (201-203):Exp.\ No.\ 731, 43--88 (1992), 1991.
\newblock S{\'e}minaire Bourbaki, Vol. 1990/91.

\bibitem[Bos99]{Bost99}
J.-B. Bost.
\newblock Potential theory and {L}efschetz theorems for arithmetic surfaces.
\newblock {\em Ann. Scient. \'{E}c. Norm. Sup.}, 32:241--312, 1999.

\bibitem[Bos04]{Bost04}
J.-B. Bost.
\newblock {Germs of analytic varieties in algebraic varieties: canonical
  metrics and arithmetic algebraization theorems.}
\newblock In {\em {A. Adolphson et al. (ed.), Geometric aspects of Dwork
  theory. Vol. I}}, pages 371--418. Walter de Gruyter, Berlin, 2004.

\bibitem[Bos15]{Bost15}
J.-B. Bost.
\newblock Theta invariants of euclidean lattices and infinite-dimensional
  hermitian vector bundles over arithmetic curves.
\newblock {\em arXiv preprint arXiv:1512.08946}, 2015.

\bibitem[BV89]{BismutVasserot89}
Jean-Michel Bismut and {\'E}ric Vasserot.
\newblock The asymptotics of the {R}ay-{S}inger analytic torsion associated
  with high powers of a positive line bundle.
\newblock {\em Comm. Math. Phys.}, 125(2):355--367, 1989.

\bibitem[CP16]{CharlesPoonen16}
Fran{\c{c}}ois Charles and Bjorn Poonen.
\newblock Bertini irreducibility theorems over finite fields.
\newblock {\em J. Amer. Math. Soc.}, 29(1):81--94, 2016.

\bibitem[Del87]{Deligne85}
P.~Deligne.
\newblock Le d\'eterminant de la cohomologie.
\newblock In {\em Current trends in arithmetical algebraic geometry ({A}rcata,
  {C}alif., 1985)}, volume~67 of {\em Contemp. Math.}, pages 93--177. Amer.
  Math. Soc., Providence, RI, 1987.

\bibitem[EW15]{ErmanWood15}
Daniel Erman and Melanie~Matchett Wood.
\newblock Semiample {B}ertini theorems over finite fields.
\newblock {\em Duke Math. J.}, 164(1):1--38, 2015.

\bibitem[Fal84]{Faltings84}
Gerd Faltings.
\newblock Calculus on arithmetic surfaces.
\newblock {\em Ann. of Math. (2)}, 119(2):387--424, 1984.

\bibitem[GS88]{GilletSoule88}
Henri Gillet and Christophe Soul{\'e}.
\newblock Amplitude arithm\'etique.
\newblock {\em C. R. Acad. Sci. Paris S\'er. I Math.}, 307(17):887--890, 1988.

\bibitem[GS91]{GilletSoule91}
H.~Gillet and C.~Soul{\'e}.
\newblock On the number of lattice points in convex symmetric bodies and their
  duals.
\newblock {\em Israel J. Math.}, 74(2-3):347--357, 1991.

\bibitem[GS92]{GilletSoule92}
Henri Gillet and Christophe Soul{\'e}.
\newblock An arithmetic {R}iemann-{R}och theorem.
\newblock {\em Invent. Math.}, 110(3):473--543, 1992.

\bibitem[Hri85]{Hriljac85}
Paul Hriljac.
\newblock Heights and {A}rakelov's intersection theory.
\newblock {\em Amer. J. Math.}, 107(1):23--38, 1985.

\bibitem[Iko15]{Ikoma15}
Hideaki Ikoma.
\newblock A {B}ertini-type theorem for free arithmetic linear series.
\newblock {\em Kyoto J. Math.}, 55(3):531--541, 2015.

\bibitem[Kah06]{Kahn06}
Bruno Kahn.
\newblock Sur le groupe des classes d'un sch\'ema arithm\'etique.
\newblock {\em Bull. Soc. Math. France}, 134(3):395--415, 2006.
\newblock With an appendix by Marc Hindry.

\bibitem[Ker11]{Kerz11}
Moritz Kerz.
\newblock Higher class field theory and the connected component.
\newblock {\em Manuscripta Math.}, 135(1-2):63--89, 2011.

\bibitem[Laz04]{Lazarsfeld04}
Robert Lazarsfeld.
\newblock {\em Positivity in algebraic geometry. {I}}, volume~48 of {\em
  Ergebnisse der Mathematik und ihrer Grenzgebiete. 3. Folge. A Series of
  Modern Surveys in Mathematics [Results in Mathematics and Related Areas. 3rd
  Series. A Series of Modern Surveys in Mathematics]}.
\newblock Springer-Verlag, Berlin, 2004.
\newblock Classical setting: line bundles and linear series.

\bibitem[LW54]{LangWeil54}
Serge Lang and Andr{\'e} Weil.
\newblock Number of points of varieties in finite fields.
\newblock {\em Amer. J. Math.}, 76:819--827, 1954.

\bibitem[Mor95]{Moriwaki95}
Atsushi Moriwaki.
\newblock Arithmetic {B}ogomolov-{G}ieseker's inequality.
\newblock {\em Amer. J. Math.}, 117(5):1325--1347, 1995.

\bibitem[Mor04]{Moriwaki04}
Atsushi Moriwaki.
\newblock The number of algebraic cycles with bounded degree.
\newblock {\em J. Math. Kyoto Univ.}, 44(4):819--890, 2004.

\bibitem[Poo04]{Poonen04}
Bjorn Poonen.
\newblock Bertini theorems over finite fields.
\newblock {\em Ann. of Math. (2)}, 160(3):1099--1127, 2004.

\bibitem[Ran06]{Randriam06}
H.~Randriambololona.
\newblock M\'etriques de sous-quotient et th\'eor\`eme de {H}ilbert-{S}amuel
  arithm\'etique pour les faisceaux coh\'erents.
\newblock {\em J. reine angew. Math.}, 590:67--88, 2006.

\bibitem[Ras95]{Raskind95}
Wayne Raskind.
\newblock Abelian class field theory of arithmetic schemes.
\newblock In {\em {$K$}-theory and algebraic geometry: connections with
  quadratic forms and division algebras ({S}anta {B}arbara, {CA}, 1992)},
  volume~58 of {\em Proc. Sympos. Pure Math.}, pages 85--187. Amer. Math. Soc.,
  Providence, RI, 1995.

\bibitem[Roq57]{Roquette57}
Peter Roquette.
\newblock Einheiten und {D}ivisorklassen in endlich erzeugbaren {K}\"orpern.
\newblock {\em Jber. Deutsch. Math. Verein}, 60(Abt. 1):1--21, 1957.

\bibitem[Ser65]{Serre65}
Jean-Pierre Serre.
\newblock Zeta and {$L$} functions.
\newblock In {\em Arithmetical {A}lgebraic {G}eometry ({P}roc. {C}onf. {P}urdue
  {U}niv., 1963)}, pages 82--92. Harper \& Row, New York, 1965.

\bibitem[Ser12]{Serre12}
Jean-Pierre Serre.
\newblock {\em Lectures on {$N_X (p)$}}, volume~11 of {\em Chapman \& Hall/CRC
  Research Notes in Mathematics}.
\newblock CRC Press, Boca Raton, FL, 2012.

\bibitem[{Sta}17]{stacks-project}
The {Stacks Project Authors}.
\newblock Stacks {P}roject.
\newblock \url{http://stacks.math.columbia.edu}, 2017.

\bibitem[Sza10]{Szamuely10}
Tam\'as Szamuely.
\newblock Corps de classes des sch\'emas arithm\'etiques.
\newblock {\em Ast\'erisque}, (332):Exp. No. 1006, viii--ix, 257--286, 2010.
\newblock S\'eminaire Bourbaki. Volume 2008/2009. Expos\'es 997--1011.

\bibitem[Voj91]{Vojta91}
Paul Vojta.
\newblock Siegel's theorem in the compact case.
\newblock {\em Ann. of Math. (2)}, 133(3):509--548, 1991.

\bibitem[Wie06]{Wiesend06}
G\"otz Wiesend.
\newblock A construction of covers of arithmetic schemes.
\newblock {\em J. Number Theory}, 121(1):118--131, 2006.

\bibitem[Yua08]{Yuan08}
Xinyi Yuan.
\newblock Big line bundles over arithmetic varieties.
\newblock {\em Invent. Math.}, 173(3):603--649, 2008.

\bibitem[YZ13]{YuanZhang13}
Xinyi Yuan and Tong Zhang.
\newblock Effective bound of linear series on arithmetic surfaces.
\newblock {\em Duke Math. J.}, 162(10):1723--1770, 2013.

\bibitem[Zha92]{Zhang92}
S.~Zhang.
\newblock Positive line bundles on arithmetic surfaces.
\newblock {\em Ann. of Math. (2)}, 136(3):569--587, 1992.

\bibitem[Zha95]{Zhang95}
S.~Zhang.
\newblock Positive line bundles on arithmetic varieties.
\newblock {\em J. Amer. Math. Soc.}, 8(1):187--221, 1995.

\end{thebibliography}

\end{document}